\documentclass[12pt,oneside]{amsart}
\usepackage{amsmath,amssymb,amsthm,mathtools,mathrsfs}
\usepackage[a4paper,width=16.2cm,top=2.5cm,bottom=2cm,
footskip=1cm]{geometry}
\usepackage{microtype}
\usepackage{needspace}

\usepackage[dvipsnames]{xcolor}

\newcommand{\handrev}[1]{#1}

\newcommand{\negrev}[1]{#1}

\newcommand{\emptyrev}[1]{#1}

\newcommand{\noterev}[1]{#1}
\newcommand{\structrev}[1]{#1}

\usepackage[colorlinks=true,linkcolor=Magenta,
citecolor=MidnightBlue,
urlcolor=MidnightBlue
]{hyperref}

\theoremstyle{plain}
\newtheorem{theorem}{Theorem}[section]
\newtheorem{assumption}[theorem]{Assumption}
\newtheorem{proposition}[theorem]{Proposition}
\newtheorem{lemma}[theorem]{Lemma}
\newtheorem{corollary}[theorem]{Corollary}
\theoremstyle{definition}

\theoremstyle{remark}
\newtheorem{remark}[theorem]{Remark}
\numberwithin{equation}{section}
\newcommand{\Hb}{H_b^0}
\newcommand{\dbarb}{\bar\partial_b}
\newcommand{\dbarG}{\bar\partial_{b,X_G}}
\newcommand{\A}{\mathcal A}
\newcommand{\B}{\mathcal B}
\newcommand{\C}{\mathcal C}
\newcommand{\I}{\mathcal I}
\newcommand{\OO}{\mathcal O}
\newcommand{\RR}{\mathbb R}
\newcommand{\CC}{\mathbb C}
\newcommand{\NN}{\mathbb N}
\newcommand{\Tt}{\mathbb T}
\DeclareMathOperator{\Dom}{Dom}
\DeclareMathOperator{\Ker}{Ker}
\DeclareMathOperator{\Ran}{Im}
\DeclareMathOperator{\coker}{coker}

\DeclareMathOperator{\Int}{Int}

\DeclareMathOperator{\supp}{supp}
\DeclareMathOperator{\Ann}{Ann}
\newcommand{\norm}[1]{\lVert#1\rVert}

\title[Quantization commutes with reduction for embeddable CR reductions]
{Quantization Commutes with Reduction 
for Embeddable CR Reductions}
\author[C.-Y. Hsiao]{Chin-Yu Hsiao}
\address{Department of Mathematics, National Taiwan University, Taipei, Taiwan}
\email{chinyuhsiao@ntu.edu.tw, chinyu.hsiao@gmail.com}
\author[R.-T. Huang]{Rung-Tzung Huang}
\address{Department of Mathematics, National Central University,
Zhongli District, \newline\mbox{\quad}\,Taoyuan City 320317, Taiwan}
\email{rthuang@math.ncu.edu.tw}
\author[X. Ma]{Xiaonan Ma}
\address{Chern Institute of Mathematics and LPMC, Nankai University,
Tianjin 300071,
\newline\mbox{\quad}\,P.\ R.\ China}
\email{xiaonan.ma@nankai.edu.cn}
\author[G. Marinescu]{George Marinescu}
\address{Department Mathematik/Informatik, Universit\"at zu K\"oln,
Weyertal 86-90, 50931 K\"oln, 
\newline\mbox{\quad}\,Germany 
\newline\mbox{\quad}\,Institute of Mathematics `Simion Stoilow', 
Romanian Academy, Bucharest, Romania}
\email{gmarines@math.uni-koeln.de}


\begin{document}
\begin{abstract}
We prove a canonical quantization-commutes-with-reduction theorem for CR
functions. Under the geometric hypotheses of Hsiao--Ma--Marinescu, if
the zero moment level is nonempty and the strictly pseudoconvex
reduction is globally CR embeddable of real dimension at least three,
the CR Guillemin--Sternberg map is a bounded isomorphism at every real
Sobolev order. The central new step is an algebraic and complex-analytic
elimination of its finite-dimensional Fredholm defect: multiplicativity
removes the kernel, while the conductor of the smooth range, integral
dependence, and holomorphic removability produce global lifts.
Conductor nonvanishing follows from the reduced Szeg\H{o} singularity;
we also prove it using global CR peak functions. 
The map is
an isomorphism of Fr\'echet algebras, the reduction is connected, and
the corresponding maps for holomorphic line bundles positive near the
zero level are isomorphisms in every nonnegative tensor degree.
Invariant holomorphic sections vanish in every negative degree;
when the reduced base has positive complex dimension, the section maps
are isomorphisms in every integer degree.

\end{abstract}
\maketitle
\tableofcontents

\section{Introduction and statement of the main results}\label{sec:introduction}

\handrev{We prove that the canonical Cauchy--Riemann (CR) 
Guillemin--Sternberg map is an
isomorphism under the natural hypotheses below.  Let $X$ be a compact
connected oriented CR
manifold with a compact group $G$-action, let $Y$ be the zero level of its
moment map, and set $X_G=Y/G$.  Assume the geometric hypotheses of
\cite{HMM}, that $Y$ is nonempty, that $\dim_{\RR}X_G\ge3$, and that
the tangential Cauchy--Riemann operator on $X_G$ has closed range.
Then restriction to $Y$ followed by descent to $X_G$ is a bounded
isomorphism from $G$-invariant CR functions on $X$ to CR functions on
$X_G$ at every real Sobolev order, with the shift $-\dim G/4$ found in
\cite[Theorem~1.2]{HMM}.  In particular, every smooth CR function on
$X_G$ has a unique smooth invariant CR lift to $X$, and $X_G$ is
connected.  We also prove that canonical restriction and descent give
isomorphisms for holomorphic sections of every nonnegative power of
an equivariant line bundle positive near the zero level.
\negrev{In negative degrees the invariant section spaces vanish, and
the canonical maps are again isomorphisms when the reduced base has
positive complex dimension.}}

{The central novelty is a mechanism that turns the Fredholm result of
\cite[Theorem~1.2]{HMM} into a canonical lifting theorem. On smooth
scalar CR functions, multiplicativity and a powers argument eliminate
the kernel. The image is a finite-codimensional subalgebra of the CR
functions on the reduction. Its conductor has no common zero, which
makes it possible to express a reduced CR function locally as a quotient
of functions that already lift. Finite codimension supplies a monic
relation that bounds these quotients near their apparent poles; on local
complex quotients, holomorphic removability then extends and glues them
to a global invariant CR lift. Thus the proof joins the algebra of CR
functions to complex-analytic extension, beyond the microlocal
Fredholm input. We also provide a proof using global CR peak functions, 
constructed with the
amended Harvey--Lawson filling theorem, the Hill--Nacinovich collar
extension, and Star\v{c}i\v{c}'s Stein-neighborhood theorem. 
}

The geometric quantization problem of Guillemin--Sternberg asks how the
quantum space of a Hamiltonian manifold compares with the quantum space
of its symplectic reduction.  For compact K\"ahler manifolds with a
positive equivariant holomorphic line bundle, their theorem identifies
the invariant holomorphic sections with sections of the reduced bundle
\cite{GS82a}.  In the symplectic setting the formulation uses equivariant
Spin$^c$ indices.  
The index formulation was proved for torus actions independently by
Meinrenken and Vergne
\cite{Meinrenken96,Vergne96}.  Meinrenken treated general compact
connected groups \cite{Meinrenken98}, and Meinrenken--Sjamaar extended
the result to singular quotients \cite{MeinrenkenSjamaar99}.
Tian--Zhang gave an analytic proof with further generalizations
\cite{TZ98}; Paradan gave a $K$-theoretic proof using transversally
elliptic symbols \cite{Paradan01}.  For compact K\"ahler manifolds,
\handrev{Teleman, Braverman, and Zhang extended the comparison to 
higher Dolbeault cohomology \cite{Teleman00,Braverman01,Zhang99}.}
For noncompact symplectic manifolds with proper moment map and a compact
group action, \handrev{Ma--Zhang \cite{MaZhang09,MZ14} settled Vergne's
conjecture, formulated in her ICM plenary lecture \cite[\S~4.3]{VergneICM06};
see also Ma's survey \cite{MaICM10}.  Paradan gave a new proof
using symplectic cuts and the wonderful compactification
\cite{Paradan11}, and Hochs--Song \cite{HochsSong17} established the
analogue of the Ma--Zhang result \cite{MaZhang09,MZ14} in the
Spin$^c$ setting of Paradan--Vergne \cite{ParadanVergne17}.}
\handrev{Ma--Zhang studied the asymptotics of the $G$-invariant Bergman
kernel and its relation to symplectic reduction \cite{MZ08};
see also the discussion in \cite{MaICM10}.}

Hsiao--Ma--Marinescu \cite{HMM} formulate the problem on a compact CR
manifold with a compact group action.  One motivation is the unit circle
bundle $X\subset L^*$ of a Hermitian holomorphic line bundle.  Its CR
functions of circle weight $m$ represent $H^0(M,L^m)$, and the Szeg\H{o}
projector organizes the Fourier components of the Bergman projections.
This viewpoint connects Toeplitz quantization \cite{BG}, the
Boutet de Monvel--Sj\"ostrand parametrix \cite{BMS}, equivariant Bergman
asymptotics \cite{MaMarinescu,MZ08}, and the $G$-invariant 
Szeg\H{o} and reduction results of Hsiao--Huang \cite{HH}.

The principal analytic result of \handrev{\cite[Theorem~1.2]{HMM}}
is the Fredholmness of the canonical restriction-and-descent map, 
with a Sobolev shift $d/4$, where
$d=\dim G$.  Its finite-dimensional kernel and Sobolev-orthogonal
cokernel consist of smooth functions.  The kernel and the Fredholm index
are independent of the Sobolev order.  These are substantial consequences
of the equivariant microlocal analysis in that paper and are used as
inputs below.

In Section~\ref{sec:proof}, we exploit the fact that the
\emph{canonical} map preserves products of smooth scalar CR functions.
The finite-dimensionality of
its kernel, together with powers of a function in that kernel, gives
injectivity.  Let $\B$ denote the algebra of smooth CR functions on
$X_G$, and let $\A\subset\B$ be the image of the smooth invariant CR
functions on $X$.
\handrev{The Fredholm and regularity results of \cite{HMM} imply that
$\B/\A$ is finite-dimensional.}  Surjectivity is proved using
the \emph{conductor} of this inclusion,
$\I=\{c\in\B:c\B\subset\A\}$.
\handrev{Since $1\in\B$, taking $b=1$ gives $c=c\cdot1\in\A$
for every $c\in\I$.}  Thus, for any $c\in\I$ and
$b\in\B$, both $c$ and $cb$ already lift to smooth invariant CR
functions on $X$.  At every point of $X_G$ some conductor
element is nonzero.  One proof uses the singularity of the reduced
Szeg\H{o} kernel; the other uses global CR peak functions and requires
the additional filling-theoretic input stated below.

The nonvanishing result provides finitely many conductor elements $c_j$
with no common zero on $X_G$.  For arbitrary $b\in\B$, both $c_j$ and
$c_jb$ have smooth invariant CR lifts $C_j$ and $B_j$ on $X$.
The quotients $B_j/C_j$ agree wherever they are defined and give the
desired lift near $Y$.  Their possible common-zero set is disjoint
from $Y$ and has empty interior.  Since $\B/\A$ is finite-dimensional,
$b$ satisfies a monic polynomial equation with coefficients in $\A$.
Lifting this equation after clearing denominators shows that the
quotients $B_j/C_j$ are locally bounded near the common-zero set.
There, local BRT coordinates for a transverse group flow identify
invariant CR functions with holomorphic functions.  Bounded
holomorphic removability extends the quotients across their apparent
poles; their agreement on the dense complement then glues the local
extensions to a smooth invariant CR function on $X$.

We now specify the geometric hypotheses of the main theorem.  
Let $(X,T^{1,0}X)$ be a compact, connected, orientable smooth CR manifold
of real dimension $2n+1$, $n\ge1$.  Write $HX$ for its real Levi
distribution and $J$ for its complex structure.  Fix a nonvanishing real
characteristic form $\omega_0$ with $\ker\omega_0=HX$.
\handrev{We use $\ker$ for pointwise kernels of bundle maps and
$\Ker$ for kernels of operators on function spaces.}
The Levi conventions are those of \cite[(2.6)--(2.7)]{HMM}:
\begin{equation}\label{eq:levi-convention}
 \mathcal L_x(v,w)=\frac12d\omega_0(v,Jw),\qquad
 \mathcal L_x(U,V)=\frac1{2i}d\omega_0(U,\overline V),
\end{equation}
for $v,w\in H_xX$ and $U,V\in T^{1,0}_xX$, respectively.

Let $G$ be a compact Lie group, not necessarily connected, with Lie
algebra $\mathfrak g$ and real dimension $d$.  In accordance with
\handrev{\cite[\S~2.3]{HMM}}, the infinitesimal vector field 
and moment map are
\begin{equation}\label{eq:moment-convention}
 \xi_X(x)=\left.\frac{d}{dt}\right|_{t=0}\exp(-t\xi)\cdot x,
 \qquad
 \langle\mu(x),\xi\rangle=\omega_0(\xi_X(x)),\quad \text{ for }
 \handrev{\xi\in\mathfrak g.}
\end{equation}
Throughout, write
\begin{equation}\label{eq:Y-XG}
 Y=\mu^{-1}(0),\qquad X_G=Y/G,
 \qquad \iota:Y\hookrightarrow X,\quad \pi:Y\longrightarrow X_G.
\end{equation}
\begin{assumption}\label{ass:HMM}
The $G$-action preserves $J$ and $\omega_0$ and is free on $Y$.  
One of the following two conditions holds:
\begin{enumerate}
\item[(i)] $\dim_{\RR}X\ge5$, and the Levi form is positive 
definite on an open neighborhood of $Y$;
\item[(ii)] $\dim_{\RR}X=3$, the Levi form is positive definite 
on all of $X$, and $\dbarb$ on functions has closed range in $L^2(X)$.
\end{enumerate}
\end{assumption}
This is the geometric assumption in \cite[Assumption~1.1]{HMM}.
In alternative~(i), Levi positivity is required only near $Y$; 
no global CR embeddability condition is imposed on $X$.
Nonemptiness of $Y$ and closed range on the \emph{reduction} 
are imposed separately in the main theorem.  
If $Y\ne\varnothing$, freeness and Levi
positivity imply that $0$ is a regular value and that
\begin{equation}\label{eq:dimension}
 \dim_{\RR}X_G=2n+1-2d=2(n-d)+1.
\end{equation}
In the three-dimensional alternative~(ii), the requirement
$\dim_{\RR}X_G\ge3$ in Theorem~\ref{thm:main} forces $d=0$.
Thus $G$ is finite, $Y=X$, and $X\to X_G=X/G$ is a finite CR covering.
The asserted isomorphism is then simply the elementary descent of
invariant CR functions through this covering, valid at every Sobolev
order with shift zero.
\handrev{We retain alternative~(ii) to match the hypotheses of
\cite{HMM}.}  The nontrivial reduction case is alternative~(i) with
$d>0$; see Remark~\ref{rem:finite}.
When this dimension is at least three, $X_G$ is a compact strictly
pseudoconvex CR manifold with induced tensors $\omega_{0,G}$
 and $J_G$; see \cite[Lemma~2.5 and Theorem~2.6]{HMM} 
 and Section~\ref{sec:prelim}.

 Restriction to the moment-map level and descent first define 
 a map on smooth invariant CR functions.  
 \handrev{To state the isomorphism at every real Sobolev order,}
 we place this map between distributional CR spaces and use 
  the Fredholm extension from \cite{HMM}.
The ordinary Sobolev spaces $H^s(X)$ and
$H^s(X_G)$, for every real $s$, are defined in
Section~\ref{subsec:sobolev}.  They include distributions when $s<0$.
Following \cite[(1.2)--(1.4)]{HMM}, put
\begin{align}
 \Hb(X)_s&=\{u\in H^s(X):\dbarb u=0\text{ distributionally}\},\notag\\
 \Hb(X)^G_s&=\{u\in\Hb(X)_s:g^*u=u\text{ for every }g\in G\},
       \label{eq:hardy-sobolev}\\
 \Hb(X_G)_s&=\{v\in H^s(X_G):\dbarG v=0\text{ distributionally}\}.
       \notag
\end{align}
Write $\Hb(X)=\Hb(X)_0$ and $\Hb(X_G)=\Hb(X_G)_0$.
On smooth invariant CR functions define
\begin{equation}\label{eq:canonical-map}
 \sigma_G=\iota_G\circ\iota^*,\qquad
 (\sigma_Gu)(\pi(y))=u(y),\quad y\in Y,
\end{equation}
where $\iota_G:\handrev{\mathscr C}^\infty(Y)^G\to 
\handrev{\mathscr C}^\infty(X_G)$ is descent.
For low Sobolev orders, the extension below is defined by density.

\begin{theorem}[\handrev{The canonical CR map is 
      an isomorphism}]\label{thm:main}
Let $X$ be compact, connected, and orientable, and assume
Assumption~\ref{ass:HMM}.  Suppose in addition that
\begin{equation}\label{eq:main-extra}
 \begin{gathered}
 Y\ne\varnothing,\qquad \dim_{\RR}X_G=2(n-d)+1\ge3,\\
 \dbarG\text{ has closed range in }L^2(X_G).
 \end{gathered}
\end{equation}
Then for every $s\in\RR$, the canonical map \eqref{eq:canonical-map}
has a unique continuous extension, which is a bounded isomorphism
\begin{equation}\label{eq:main-isomorphism}
 \sigma_{G,s}:\Hb(X)^G_s\overset{\sim}{\longrightarrow}
                  \Hb(X_G)_{s-d/4}.
\end{equation}
For each $s$, there is a constant $C_s\ge1$ such that
\begin{equation}\label{eq:main-estimate}
 C_s^{-1}\norm{u}_{s}\le
 \norm{\sigma_{G,s}u}_{X_G,s-d/4}
 \le C_s\norm{u}_{s},\qquad u\in\Hb(X)^G_s.
\end{equation}
On smooth functions, $\sigma_G$ is a unital algebra isomorphism 
and a Fr\'echet-space isomorphism.  Moreover, $X_G$ is connected.
\end{theorem}

\handrev{Note that the map \eqref{eq:canonical-map} does not depend on
the metrics on $TX$ and $TX_G$, and the Sobolev norms for different
smooth metrics are equivalent.  Thus we may work with the metrics and
volume forms fixed as in \cite[Convention~2.8]{HMM}; see
Section~\ref{subsec:metrics}.\par}

For a compact strictly pseudoconvex manifold of real dimension 
at least three, closed range on functions is equivalent to global CR
embeddability, see \noterev{\cite{BdM75,BMS,Bur79,HL75,Kohn86}
and the attribution of the two implications in 
Section~\ref{subsec:fredholm}}; this is the equivalence recalled in
\handrev{\cite[\S~2.4]{HMM}}. Consequently, \eqref{eq:main-extra}
can be formulated with embeddability of $X_G$ in place of 
closed range. It is automatic in reduced real dimension at least
five by the embedding theorem of Boutet de Monvel \cite{BdM75}, 
whereas in reduced dimension three, it
is a substantive condition.  Although $X$ is assumed connected,
 neither Assumption~\ref{ass:HMM} nor \eqref{eq:main-extra}
requires $X_G$ to be connected; its connectedness is 
a conclusion of Theorem~\ref{thm:main}.

\structrev{This connectedness conclusion is a CR analogue of the
connectedness theorems for moment-map fibres in compact Hamiltonian
geometry, due to Atiyah, Guillemin--Sternberg, and Kirwan
\cite{Atiyah82,GS82convexity,Kirwan84}.
Here it follows from the algebra isomorphism on smooth CR functions
and the absence of nonconstant idempotents on the connected manifold
$X$.  When $G$ is connected, the principal bundle $Y\to X_G$ also
implies that the zero level $Y=\mu^{-1}(0)$ is connected.}

\emptyrev{When $Y=\varnothing$, every $G$-invariant CR distribution
$u$ is smooth by the local complex-quotient argument of
Lemma~\ref{lem:quotient} and Remark~\ref{rem:invariant-away}.
Choose $x_0$ where $|u|$ attains its maximum, and put $c=u(x_0)$.
At every point of the nonempty closed set $E=\{x\in X:u(x)=c\}$,
the local holomorphic representative attains that maximum; the
maximum-modulus principle makes it constant, so $E$ is also open.
Connectedness of $X$ gives $u\equiv c$.  Consequently,
\[
 \Hb(X)^G_s=\CC\qquad\text{for every }s\in\RR.
\]
Thus the nonempty-level hypothesis isolates the genuinely 
reduced case.}

The first application concerns Sasakian Reeb weights. The 
canonical CR isomorphism removes the finite exceptional set left in
\cite[Theorem~1.3]{HMM}.
\handrev{In the setting of the following corollary, let $\Tt$ be the
closure of the Reeb flow in the isometry group of the Sasakian metric.
This compact torus preserves $Y$ and induces a CR action on $X_G$.
Write $\widehat{\Tt}=\operatorname{Hom}_{\mathrm{cont}}(\Tt,S^1)$ 
for its character group.  
For $Z=X,X_G$ and $\alpha\in\widehat{\Tt}$, define
\[
 H^0_{b,\alpha}(Z)
 =\{u\in L^2(Z):\bar\partial_{b,Z}u=0,\quad
          t^*u=\alpha(t)u\text{ for every }t\in\Tt\},
\]
where the CR equation is understood distributionally and
$(t^*u)(z)=u(t^{-1}\cdot z)$.
These spaces are finite-dimensional and
consist of smooth functions; equivalently, one may replace $L^2(Z)$
by $\mathscr C^\infty(Z)$ in the definition.
The superscript $G$ denotes the subspace of $G$-invariant functions.}

\begin{corollary}[\handrev{Canonical CR isomorphism at 
      every Reeb weight}]\label{cor:sasaki}
Let $X$ be a compact connected Sasakian manifold, and 
let $G$ preserve its CR
\noterev{structure and contact form $\omega_0$, and act freely on
$Y\ne\varnothing$.}  Suppose $\dim_{\RR}X_G\ge3$.
\handrev{For every $\alpha\in\widehat{\Tt}$, 
the canonical restriction gives}
\begin{equation}\label{eq:weight-map}
 \sigma_G:H^0_{b,\alpha}(X)^G\overset{\sim}{\longrightarrow}
                         H^0_{b,\alpha}(X_G).
\end{equation}
\end{corollary}
\handrev{The proof is given in Section~\ref{subsec:sasaki-proof}.}
The holomorphic line-bundle result holds in every nonnegative tensor
degree\negrev{, and in every integer degree when the reduced base has
positive complex dimension}.  
Its moment map uses the Kostant convention of
\cite[(6.1)]{HMM}; its relationship with \eqref{eq:moment-convention}
is checked in Section~\ref{subsec:circle}.
\begin{theorem}[\handrev{Holomorphic quantization and 
      negative-degree vanishing}]
\label{thm:line}
Let $M$ be a compact connected complex manifold of dimension $n\ge1$,
and let $(L,h^L)$ be a Hermitian holomorphic line bundle.  
Let a compact Lie group $G$ act holomorphically on $M$ 
and lift holomorphically to $L$, preserving $h^L$.  
For $\xi\in\mathfrak g$, use
$\xi_M=\left.\frac{d}{dt}\right|_0\exp(-t\xi)\cdot x$, let
$\mathcal L_\xi$ be the infinitesimal action on sections, and define
\begin{equation}\label{eq:Kostant}
 2\pi i\langle\widetilde\mu,\xi\rangle
    =\nabla^L_{\xi_M}-\mathcal L_\xi.
\end{equation}
\handrev{Assume that $G$ acts freely on
$P=\widetilde\mu^{-1}(0)\ne\varnothing$, and that $iR^L$ is positive}
definite on a neighborhood of $P$.  Let
\[
 M_G=P/G,\qquad L_G=(L|_P)/G
\]
carry their reduced holomorphic structures.  
Then $M_G$ is connected,
and canonical restriction and descent induce isomorphisms
\begin{equation}\label{eq:line-all}
 A_m:H^0(M,L^m)^G\overset{\sim}{\longrightarrow}H^0(M_G,L_G^m),
       \qquad m\in\handrev{\NN}=\{0,1,2,\ldots\}.
\end{equation}
These maps form an isomorphism of the corresponding 
nonnegatively graded section rings.

\negrev{For $m<0$, set $L^m=(L^*)^{-m}$ and
$L_G^m=(L_G^*)^{-m}$.  The invariant section space $H^0(M,L^m)^G$
vanishes in every negative degree, irrespective of the dimension of
$M_G$.  If $\dim_\CC M_G\ge1$, canonical restriction and descent are
isomorphisms $A_m$ for every $m\in\mathbb Z$; in negative degrees,
\begin{equation}\label{eq:negative-vanishing}
 H^0(M,L^m)^G=H^0(M_G,L_G^m)=0,\qquad m<0.
\end{equation}}
\end{theorem}
Note that in Theorem~\ref{thm:line}, no global positivity or
semipositivity of $iR^L$, and no global K\"ahler assumption
on $M$, is required.

\structrev{A canonical holomorphic section isomorphism 
identifies the
section spaces by restriction and descent, whereas an index identity
compares virtual dimensions, expressed in the holomorphic setting as
alternating sums of cohomology dimensions.}
The present argument does not prove an
all-degree version of the almost-complex Spin$^c$ statement in
\cite[Theorem~6.1]{HMM}.  
When $\dim_{\RR}X_G=1$, restriction and descent need not take values in the
componentwise disk Hardy spaces, so this case must be excluded from
\cite[Theorem~1.2]{HMM}; Section~\ref{sec:examples}
 gives an explicit example,
and Remark~\ref{rem:dim-one} locates the gap in \cite{HMM}.

\structrev{When $\dim_\CC M_G=0$, the associated 
reduced circle bundle
$X_G$ is one-dimensional and Theorem~\ref{thm:main} does not apply;
the nonnegative-degree isomorphisms in Theorem~\ref{thm:line} follow
instead by stabilization (Section~\ref{subsec:powers}).}

This paper is organized as follows.
Section~\ref{sec:prelim} fixes the CR and Sobolev conventions
and lists the precise imported analytic results.  
Section~\ref{sec:proof} proves Theorem~\ref{thm:main}, and
Section~\ref{sec:applications} proves Corollary~\ref{cor:sasaki}
and Theorem~\ref{thm:line}.  
Section~\ref{sec:examples} gives examples and
discusses the one-dimensional CR obstruction.

\section{Geometric and analytic preliminaries}\label{sec:prelim}
The proof of Theorem~\ref{thm:main} uses
 \cite[Theorem~1.2]{HMM} and its accompanying regularity 
results together with compatible conventions for CR reduction 
and Sobolev spaces. These preliminaries fix that framework and 
record the analytic properties of the reduction map used later.  
The canonical restriction-and-descent map is also distinguished 
from its normalized counterpart, since only
the former preserves products of smooth CR functions.

\subsection{CR structures, characteristic forms and CR reduction}
We begin by fixing the characteristic form, which determines the
Levi-sign convention.  Its normalization of characteristic fields
will later be distinguished from the group-generated transverse
fields used in the local quotient argument.
A CR structure of hypersurface type on $X$ is a rank-$n$ complex
subbundle $T^{1,0}X\subset\CC TX$
\handrev{, the complexified tangent bundle,} such that
\[
 \begin{gathered}
 T^{1,0}X\cap T^{0,1}X=\{0\},\qquad
 T^{0,1}X=\overline{T^{1,0}X},\\
 [\handrev{\mathscr C}^\infty(X,T^{1,0}X),
 \handrev{\mathscr C}^\infty(X,T^{1,0}X)]
       \subset \handrev{\mathscr C}^\infty(X,T^{1,0}X).
 \end{gathered}
\]
The real bundle $HX$ is characterized by
$\CC HX=T^{1,0}X\oplus T^{0,1}X$, and $J$ acts by $i$ and $-i$ on
these two summands.  The characteristic conormal bundle is
$(HX)^0=\RR\omega_0$.  Orientability of $X$, together with 
the complex orientation of $HX$, permits this global choice.

The form $d\omega_0|_{HX}$ is of type $(1,1)$ by CR integrability.
Under convention \eqref{eq:levi-convention}, 
positivity of $\mathcal L$ is the positivity of 
$d\omega_0(\cdot,J\cdot)$ on $HX$. A characteristic
vector field $T$ is a real vector field satisfying
\begin{equation}\label{eq:characteristic}
 \omega_0(T)=1,\qquad
 \CC TX=T^{1,0}X\oplus T^{0,1}X\oplus\CC T.
\end{equation}
\handrev{This is the normalization of \cite[(2.8)--(2.9)]{HMM}.
Section~\ref{sec:local} uses locally chosen group-generated 
transverse CR fields $V_\xi=\xi_X$.}
If $\omega_0$ is contact,
its Reeb field $R$ is additionally characterized by 
$\iota_Rd\omega_0=0$.

We work under Assumption~\ref{ass:HMM}.  By
\cite[Lemma~2.5]{HMM}, the moment map $\mu:X\to\mathfrak g^*$ in
\eqref{eq:moment-convention} is regular at $0$.
By \cite[Theorem~2.6]{HMM}, the reduced space $X_G=Y/G$ is a strictly
pseudoconvex CR manifold with contact form $\omega_{0,G}$ satisfying
\begin{equation}\label{eq:reduced-form}
 \iota^*\omega_0=\pi^*\omega_{0,G}.
\end{equation}
We recall the induced complex structure $J_G$ on
$HX_G=\ker\omega_{0,G}$.
For $x\in X$, set
$\underline{\mathfrak g}_x
=\operatorname{Span}\{\xi_X(x):\xi\in\mathfrak g\}$.
Since $G$ acts freely on $Y$, these vector spaces form a vector 
bundle $\underline{\mathfrak g}$ near $Y=\mu^{-1}(0)$.
Write $\mathfrak g_Y=\underline{\mathfrak g}|_Y$.

On a neighborhood of $Y$, $g^{HX}=d\omega_0(\cdot,J\cdot)$ is a
positive $J$-invariant metric.  Write $T^HY$ for the complement of
$\mathfrak g_Y$ in $TY\cap HX$ relative to this metric.  The
splittings and invariance of \cite[(2.34)--(2.37)]{HMM} are
\begin{equation}\label{eq:reduction-splittings}
 \begin{split}
 TX|_Y&=TY\oplus J\mathfrak g_Y,\\
 TY\cap HX&=T^HY\oplus\mathfrak g_Y,\\
 HX|_Y&=T^HY\oplus\mathfrak g_Y\oplus J\mathfrak g_Y,
 \qquad JT^HY=T^HY.
 \end{split}
\end{equation}
The map
$d\pi:T^HY\to\pi^*HX_G$ is an isomorphism; transporting $J$ gives
$J_G$ via $(J_Gv)^H=Jv^H$\handrev{, where $v^H\in T^HY$ is defined
by $d\pi(v^H)=v$.}
This also proves that \eqref{eq:canonical-map} sends smooth 
invariant CR functions to smooth CR functions on $X_G$.

\subsection{Weak CR equations and ordinary Sobolev spaces }
\label{subsec:metrics}
\label{subsec:sobolev}
\handrev{Under Assumption~\ref{ass:HMM}, we use the adapted
$G$-invariant Hermitian metric $g^{TX}$ of
\cite[Lemma~2.7 and Convention~2.8]{HMM}:}
The decomposition in \eqref{eq:characteristic} is orthogonal,
\handrev{$|T|=1$}, and the orbit directions are orthogonal 
to the reduced horizontal directions on $Y$.
\handrev{It determines the $L^2$ pairings and weak 
operator realizations
used in the Fredholm statement, as well as the ordinary Sobolev norms
below.}  These choices induce the reduced metric
and the smooth volume densities $dv_X$ and $dv_{X_G}$.  The inner
products are linear in the first variable:
\begin{equation}\label{eq:L2-products}
 (u,v)=\int_Xu\overline v\,dv_X,\qquad
 (a,b)_{X_G}=\int_{X_G}a\overline b\,dv_{X_G}.
\end{equation}
The same convention is used for differential forms and bundle-valued
sections.  Write $L^2_{(0,q)}(X)$ for the space of square-integrable
sections of $T^{*0,q}X=\Lambda^qT^{*0,1}X$.

For smooth $u$, the tangential Cauchy--Riemann operator is given by
$(\dbarb u)(\overline Z)=\overline Zu$ for
$\overline Z\in T^{0,1}X$.  Its weak maximal extension is
\begin{equation}\label{eq:weak-maximal}
 \begin{split}
 \dbarb &: \Dom\dbarb\subset L^2(X)\longrightarrow L^2_{(0,1)}(X),\\
 \Dom\dbarb&=\{u\in L^2(X):\dbarb u\in L^2_{(0,1)}(X)
                               \text{ distributionally}\}.
 \end{split}
\end{equation}
It is a densely defined closed operator.  
The analogous definitions on
$X_G$ give $\dbarG$. Consequently, $\Hb(X)$ and $\Hb(X_G)$ are
closed Hilbert subspaces.  The orthogonal projections are denoted by
\begin{equation}\label{eq:szego-projections}
 \begin{gathered}
 S:L^2(X)\to\Hb(X),\qquad S_G:L^2(X)\to\Hb(X)^G,\\
 S_{X_G}:L^2(X_G)\to\Hb(X_G),
 \end{gathered}
\end{equation}
in agreement with \cite[(1.15)--(1.17)]{HMM}.
\handrev{In the following we use the ordinary elliptic Sobolev scale.}  Let $\Delta^X\ge0$ and
$\Delta^{X_G}\ge0$ be
the Laplace--Beltrami operators of the fixed metrics.  Following
\cite[(5.1)--(5.3)]{HMM}, put
\begin{equation}\label{eq:Lambda}
 \Lambda_s=(1+\Delta^X)^{s/2},\qquad
 \widehat\Lambda_s=(1+\Delta^{X_G})^{s/2},\quad s\in\RR.
\end{equation}
These operators are defined by spectral functional calculus on
distributions.  Set
\begin{align}
 H^s(X)&=\{u\in\mathcal D'(X):\Lambda_su\in L^2(X)\},
 & (u,v)_s&=(\Lambda_su,\Lambda_sv),\label{eq:sobolev-X}\\
 H^s(X_G)&=\{a\in\mathcal D'(X_G):\widehat\Lambda_sa\in L^2(X_G)\},
 & (a,b)_{X_G,s}&=(\widehat\Lambda_sa,\widehat\Lambda_sb)_{X_G}.
 \notag
\end{align}
\handrev{Write $\norm{\cdot}_s$ and $\norm{\cdot}_{X_G,s}$ for the
corresponding norms.  In particular, $H^0=L^2$, and negative-order
spaces contain distributions.}  We use the inclusions and smoothness
criterion
\begin{equation}\label{eq:nesting}
 H^{s_1}(X)\hookrightarrow H^{s_2}(X)\quad(s_1\ge s_2),\qquad
 \bigcap_{k\ge0}H^{s_0+k}(X)=\handrev{\mathscr C}^\infty(X).
\end{equation}
The analogous statements hold on $X_G$.  
Different smooth metrics give equivalent Sobolev norms.

The spaces in \eqref{eq:hardy-sobolev} are closed subspaces of these
Hilbert spaces.  Indeed, $\dbarb:H^s\to H^{s-1}_{(0,1)}(X)$ is
continuous, and $g^*:H^s\to H^s$ is continuous for each fixed 
\handrev{$g\in G$}. Thus, the CR equation and the invariance 
conditions are closed conditions.

\subsection{Fredholm properties of the canonical CR map}
\label{subsec:fredholm}

For a compact strictly pseudoconvex CR manifold of dimension at least
three, \handrev{\cite[\S~2.4]{HMM} recall} that global 
CR embeddability and closed
\handrev{range of $\dbarb$ on functions are equivalent.  
We apply this
equivalence componentwise to $X_G$.  In dimension at least five,
embeddability follows from \cite{BdM75}; in dimension three, closed
range is part of our hypotheses.}
\noterev{Closed range implies embeddability by the CR embedding
construction of Boutet de Monvel and Burns \cite{BdM75,Bur79}, using
the regularity and local structure of the Szeg\H{o} projection
\cite{BMS}.  Conversely, an embedded $X_G$ is the strictly
pseudoconvex boundary of a compact complex manifold, obtained from
the Harvey--Lawson filling \cite{HL75,HLadd} by resolving its finitely
many interior singularities; see \cite[\S~2.4]{HMM} and
\cite[\S~1]{MY07}.  Kohn's closed-range theorem \cite{Kohn86}
applies to such boundaries.}

The following formulation collects the consequences of \cite{HMM}
needed below for intrinsic CR targets of real dimension at least three.
\begin{theorem}[Hsiao--Ma--Marinescu]\label{thm:HMM-input}
Assume the geometric hypotheses in Assumption~\ref{ass:HMM}.
\handrev{Suppose $\dim_{\RR}X_G\ge3$, and that}
$\dbarG$ has closed range in $L^2$.  Then the following hold.
\begin{enumerate}
\item[(a)] For every $s\in\RR$, smooth invariant CR functions are
dense in $\Hb(X)^G_s$ in the $H^s(X)$ topology.  
The canonical map has a bounded Fredholm
extension $\sigma_{G,s}:\Hb(X)^G_s\to\Hb(X_G)_{s-d/4}$.
\item[(b)] Its kernel is a finite-dimensional subspace of
$\handrev{\mathscr C}^\infty(X)\cap\Hb(X)^G$, independent of $s$.
\item[(c)] With respect to the reduced Sobolev inner product
defined in Section~\ref{subsec:sobolev},
\begin{equation}\label{eq:orthogonal-cokernel}
 (\Ran\sigma_{G,s})^\perp
 =\{v\in\Hb(X_G)_{s-d/4}:(\sigma_{G,s}u,v)_{X_G,s-d/4}=0\ 
 \handrev{\text{for all }}u\in\Hb(X)^G_s\}
\end{equation}
is a finite-dimensional space of smooth reduced CR functions.  Its
\emph{dimension}, and hence the index, is independent of $s$.
\item[(d)] $S_G$ and $S_{X_G}$ preserve smoothness.  
In particular, $\Hb(X_G)\cap \handrev{\mathscr C}^\infty(X_G)$ 
is dense in $\Hb(X_G)$.
\end{enumerate}
\end{theorem}
The density assertion in (a) is stated in
\cite[Corollary~5.2]{HMM}.  It follows there from smoothness 
preservation by $S_G$ (Theorem~3.27) and its boundedness on
$H^s(X)$ for every real $s$ (Theorem~5.1 applied to $S_G$).
These invariant-projector results apply under
Assumption~\ref{ass:HMM}(i), with Levi positivity only near $Y$ and
without a closed-range assumption for the full operator $\dbarb$ 
on $X$. For the remaining assertions, boundedness in (a) is
\cite[Theorem~5.3]{HMM}, and Fredholmness is \cite[Theorem~1.2]{HMM};
(b) is \cite[Theorem~5.4]{HMM}; smooth finite-dimensionality 
in (c) is
\cite[Theorem~5.5]{HMM}, and independence of its dimension is
\cite[Theorem~5.6]{HMM}.  For (d), use \cite[Theorem~3.27]{HMM} for
$S_G$ and \cite[(1.17)]{HMM} for $S_{X_G}$.

Uniqueness of continuous extension gives compatibility.  More
explicitly, if $s_1\ge s_2$ and $u\in\Hb(X)^G_{s_1}$, take
$u_\nu\in \handrev{\mathscr C}^\infty(X)\cap\Hb(X)^G$ 
with $u_\nu\to u$ in
$H^{s_1}(X)$.  The same convergence holds in $H^{s_2}(X)$, and
\[
 \norm{\sigma_Gu_\nu-\sigma_{G,s_1}u}_{X_G,s_1-d/4}
 \le C_{s_1}\norm{u_\nu-u}_{s_1}\longrightarrow0.
\]
The same sequence converges to $\sigma_{G,s_2}u$ at the lower target
order.  Thus
\begin{equation}\label{eq:compatibility}
 \sigma_{G,s_1}u=\sigma_{G,s_2}u
 \quad\text{as distributions on }X_G.
\end{equation}

Theorem~\ref{thm:HMM-input} invokes the direct
assertion of \cite[Theorem~1.2]{HMM} for the canonical map
$\sigma_G=\iota_G\circ\iota^*$ and its Sobolev extensions
$\sigma_{G,s}:\Hb(X)^G_s\to\Hb(X_G)_{s-d/4}$, $s\in\RR$,
defined in \cite[(1.5)--(1.6)]{HMM}.
The normalized projected operator of \cite[(1.20)]{HMM} is a different
auxiliary map; multiplicativity in this paper always concerns
$\sigma_G$ on smooth invariant CR functions.

\section{Bijectivity of the canonical CR reduction map}\label{sec:proof}

Theorem~\ref{thm:HMM-input} leaves a finite-dimensional
kernel and cokernel as possible obstructions to bijectivity.
Multiplicativity and connectedness
of $X$ eliminate the kernel.  Regularity of the Fredholm complements
then identifies the smooth image with a finite-codimensional subalgebra
of the reduced CR functions, providing the starting point for the
conductor argument.
Throughout Section~\ref{sec:proof}, the
hypotheses of Theorem~\ref{thm:main} are in force. 

This section is organized as follows.
In Section \ref{sec:range}, we prove the injectivity of $\sigma_G$
and introduce the conductor of the smooth range algebra. 
 In Section \ref{subsec:szego-proof}, we prove that its conductor 
has no common zero on $X_G$, stated as 
Lemma~\ref{lem:nocommon}, using
the Szeg\H{o} kernel. In Section \ref{subsec:peaks}, we provide an 
alternative proof by using global CR peak functions.  
In Section \ref{sec:local}, we explain the BRT coordinate, which gives 
a local complex quotient on which the invariant CR function becomes holomorphic.
In Section \ref{sec:surjectivity}, we construct a lift
 of a reduced CR function.  We first form that lift
where suitable conductor denominators do not vanish, then remove its
apparent poles. In Section \ref{s:proofSob}, 
we pass from smooth surjectivity to the full Sobolev
isomorphism as stated in Theorem \ref{thm:main}. 
The resulting algebra isomorphism also yields
the connectedness of $X_G$. 

\subsection{Injectivity of $\sigma_G$ and the smooth range algebra}\label{sec:range}
Set
\begin{equation}\label{eq:algebras}
 \C=\handrev{\mathscr C}^\infty(X)\cap\Hb(X)^G,\qquad
 \B=\handrev{\mathscr C}^\infty(X_G)\cap\Hb(X_G),\qquad
 \A=\sigma_G(\mathcal{C})\subset\B.
\end{equation}
The canonical map is a unital homomorphism of commutative complex
algebras.

\begin{proposition}[Injectivity]\label{prop:injective}
For every $s\in\RR$, $\Ker\sigma_{G,s}=\{0\}$.
\end{proposition}
\begin{proof}
\handrev{Take $u\in\Ker\sigma_{G,s}$. By 
Theorem~\ref{thm:HMM-input}(b), $u$ is smooth and $u|_Y=0$.}
For every integer $k\ge1$, the Leibniz rule and invariance give
$u^k\in\C$, and multiplicativity gives $\sigma_G(u^k)=0$.
Since every smooth function belongs to $H^s(X)$, all the powers lie
in the same finite-dimensional kernel.  Some nontrivial finite linear
combination of them is zero, that is, $P(u)=0$ 
for a nonzero polynomial
$P\in\CC[t]$.  The continuous image $u(X)$ is a connected subset of
the finite zero set of $P$, so $u$ is constant.  
Since $Y$ is nonempty and $u|_Y=0$, the constant is zero.
\end{proof}
The map $\sigma_G:\C\to\A$ is therefore an 
algebra isomorphism.
For each $k=0,1,2,\ldots$, consider the map at Sobolev order $d/4+k$:
\begin{equation}\label{eq:Tk-F}
 \begin{gathered}
 T_k=\sigma_{G,d/4+k}:\Hb(X)^G_{d/4+k}\longrightarrow\Hb(X_G)_k,
 \quad k\in\NN,\\
 \mathcal R_0=\Ran T_0,\qquad F=\mathcal R_0^\perp\subset\Hb(X_G).
 \end{gathered}
\end{equation}
\handrev{The orthogonal complement is taken in the Hardy space 
with its $L^2$ inner product.}  Fredholmness gives a closed
\handrev{range, and \cite[Theorem~5.5]{HMM}, applied at $s=d/4$
(target order zero), gives the $L^2$ orthogonal decomposition }
\begin{equation}\label{eq:F-smooth}
 \Hb(X_G)=\mathcal R_0\oplus F,\qquad F\subset\B,
 \qquad r:=\dim_\CC F<\infty.
\end{equation}

\begin{lemma}[Smooth range]\label{lem:smoothrange}
One has $\A=\B\cap\mathcal R_0$ and the $L^2$ orthogonal decomposition 
$\B=\A\oplus F$.
\end{lemma}
\begin{proof}
By \eqref{eq:compatibility}, $T_ku=T_0u$ after the target inclusion,
so $\Ran T_k\subset\mathcal R_0$.  Since $F$ is smooth,
$F\subset\Hb(X_G)_k$, and $F\cap\Ran T_k=0$.  Consider the map
induced by the quotient projection,
\[
 F\longrightarrow\Hb(X_G)_k/\Ran T_k.
\]
\handrev{It is injective.  By \cite[Theorem~5.6]{HMM}, applied at
$s=d/4+k$, its target has dimension $r$ for every $k\in\NN$.
This constancy also follows from the independence of both the kernel
and the index in \cite[Theorem~1.2]{HMM}.  Since $\dim F=r$, the map
is an isomorphism.  Therefore}
\begin{equation}\label{eq:Tk-complement}
 \Hb(X_G)_k=\Ran T_k\oplus F.
\end{equation}
\handrev{This is a vector-space direct sum.}

Take $b\in\B\cap\mathcal R_0$.  It belongs to every
$\Hb(X_G)_k$.  By \eqref{eq:Tk-complement} write
$b=T_ku_k+f_k$, with $f_k\in F$.  Both $b$ and $T_ku_k$ lie in
$\mathcal R_0$, so $f_k\in F\cap\mathcal R_0=0$.
Consequently $T_ku_k=b$ for all $k$.  Compatibility and
Proposition~\ref{prop:injective} imply that all $u_k$ coincide as a
single distribution $u$.  Since $u\in H^{d/4+k}(X)$ for every $k$,
\eqref{eq:nesting} gives $u\in \handrev{\mathscr C}^\infty(X)$.  It is invariant and
CR, hence $b\in\A$.  The opposite inclusion is immediate.
Finally the $L^2$ decomposition of a smooth $b$ in
$\mathcal R_0\oplus F$ has a smooth $F$ component, and therefore a
smooth $\mathcal R_0$ component.  The first assertion puts the latter
in $\A$.
\end{proof}

Choose an $L^2$-orthonormal basis $e_1,\ldots,e_r$ of $F$ and 
by Lemma~\ref{lem:smoothrange}, we have
\begin{equation}\label{eq:lambda}
\A=\bigcap_{i=1}^r\Ker\lambda_i\quad\text{inside }\B,
\quad \text{  with   }  \lambda_i(b)=(b,e_i)_{X_G}.
\end{equation}
If $r=0$, smooth surjectivity is already proved.  
The conductor arguments below also cover this case, 
since then $\A=\B$ and $1$ belongs to the conductor.

Define the conductor of the algebra inclusion $\A\subset\B$ by
\begin{equation}\label{eq:conductor}
 \I=\{c\in\B:c\B\subset\A\}.
\end{equation}
Since $1\in\B$, every $c\in\I$ belongs to $\A$.  Addition and
additive inverses preserve the condition in \eqref{eq:conductor}.
If $c\in\I$ and $b,f\in\B$, then $(bc)f=c(bf)\in\A$, so
$bc\in\I$.  Thus $\I$ is an ideal of $\B$ contained in $\A$.
If $J_0$ is any ideal of $\B$ contained in $\A$, then
$c\B\subset J_0\subset\A$ for $c\in J_0$, hence $J_0\subset\I$.
This proves that $\I$ is the largest such ideal with respect to
inclusion.  Equivalently, it is $\Ann_{\A}(\B/\A)$ when the quotient
is regarded as an $\A$-module.

Put $b_0=1$ and $b_j=e_j$ for $1\le j\le r$.  Then
\begin{equation}\label{eq:conductor-equations}
 \I=\bigcap_{\substack{1\le i\le r\\0\le j\le r}}
       \Ker\bigl(c\longmapsto\lambda_i(cb_j)\bigr)
       \quad\text{inside }\B.
\end{equation}
A proof of \eqref{eq:conductor-equations} is as follows. 
By (\ref{eq:lambda}), $j=0$ forces $c\in\A$, and the other 
equations give $ce_j\in\A$.  Writing any $b\in\B$ as
$a+\sum_j\gamma_je_j$, with $a\in\A$ and $\gamma_j\in\CC$,
then proves $cb\in\A$.  The converse is the definition of the
conductor.  

The finitely many defining functionals have
smooth integral representatives:
\begin{equation}\label{eq:lambda-products}
 \lambda_i(cb_j)=(cb_j,e_i)_{X_G}
                    =(c,\overline{b_j}e_i)_{X_G}.
\end{equation}

\subsection{Nonvanishing of conductors via the Szeg\H{o} kernel.}
\label{subsec:szego-proof}
 The conductor equations above provide the algebraic part of the
 nonvanishing argument.  A hypothetical common zero would produce a
 smooth representative of point evaluation; ruling out such a
 representative is the remaining task.

\begin{lemma}[No common zero]\label{lem:nocommon}
For every $p\in X_G$ there is $c\in\I$ such that $c(p)\ne0$.
Consequently there are finitely many $c_1,\ldots,c_N\in\I$ with no
common zero on $X_G$.
\end{lemma}
\structrev{We first prove Lemma~\ref{lem:nocommon} using the
Szeg\H{o} kernel; Subsection~\ref{subsec:peaks} gives a peak-function
alternative.}
To prepare both proofs, suppose instead that every $c\in\I$ vanishes
at a fixed $p$.  Define
\[
 L:\B\longrightarrow\CC^{r(r+1)},\qquad
 L(b)=(\lambda_i(bb_j))_{i,j}.
\]
The functional $E_p(b)=b(p)$ vanishes on $\Ker L=\I$, so it
\handrev{factors through $\Ran L$.  
Extending this functional to the ambient
finite-dimensional space gives constants $\alpha_{ij}$ such that}
\begin{equation}\label{eq:evaluation-representative}
 b(p)=\sum_{i,j}\alpha_{ij}\lambda_i(bb_j)=(b,h)_{X_G},\qquad
 h=\sum_{i,j}\overline{\alpha_{ij}}\,\overline{b_j}e_i
                    \in \handrev{\mathscr C}^\infty(X_G).
\end{equation}

\begin{proof}[First proof of Lemma~\ref{lem:nocommon}]
Suppose that $p$ is a common zero of $\I$, and let $h$ be the smooth
function obtained in \eqref{eq:evaluation-representative}.  We show
that no such representative of point evaluation can exist.
\structrev{Since $X_G$ is compact and strictly pseudoconvex and
$\bar\partial_{b,X_G}$ has closed range, the Boutet de
Monvel--Sj\"ostrand construction \cite{BMS} applies directly to its
Szeg\H{o} projector.
Write $\dim_{\RR}X_G=2q+1$ with $q=n-d\ge1$; near $p$, its kernel
has the form}
\begin{equation}\label{eq:szego-parametrix}
 S_{X_G}(x,y)=\int_0^\infty e^{it\phi(x,y)}\chi(t)a(x,y,t)\,dt
                                      +R(x,y),
\end{equation}
where the integral is oscillatory, $R$ is smooth, $\chi=0$ for
$t\le1$ and $\chi=1$ for $t\ge2$, and $a$ is a classical symbol
of order $q$.  On a smaller precompact coordinate neighborhood,
\begin{equation}\label{eq:szego-leading}
 \begin{gathered}
 a(x,y,t)=a_0(x,y)t^q+r(x,y,t),\qquad a_0(p,p)\ne0,\\
 |r(x,y,t)|\le Ct^{q-1}\quad(t\ge1).
 \end{gathered}
\end{equation}

Apply \eqref{eq:evaluation-representative} to $b=S_{X_G}\varphi$ for
arbitrary $\varphi\in \handrev{\mathscr C}^\infty(X_G)$.  
By smoothness preservation and self-adjointness,
\begin{equation}\label{eq:projected-evaluation}
 (S_{X_G}\varphi)(p)
 =(S_{X_G}\varphi,h)_{X_G}
 =(\varphi,S_{X_G}h)_{X_G},\qquad S_{X_G}h\in 
 \handrev{\mathscr C}^\infty(X_G).
\end{equation}
For $\varphi\in \handrev{\mathscr C}_c^\infty(X_G\setminus\{p\})$, 
off-diagonal
smoothness justifies the ordinary kernel formula.  Comparing it with
\eqref{eq:projected-evaluation} for all such test functions yields
\begin{equation}\label{eq:offdiag-evaluation}
 S_{X_G}(p,y)=\overline{S_{X_G}h(y)},\qquad y\ne p.
\end{equation}
The right side stays bounded as $y\to p$.
To contradict this boundedness, it suffices to find a path 
approaching
$p$ along which the leading term of \eqref{eq:szego-parametrix}
dominates its remainder.  The estimate on $r$ in
\eqref{eq:szego-leading} is used for precisely this comparison.

For the phase in this construction,
$d_y\operatorname{Re}\phi(p,p)$ is a nonzero multiple of the 
contact form, and the quadratic part of 
$\operatorname{Im}\phi(p,y)$ is positive on $H_pX_G$.
Thus $\{y:\operatorname{Re}\phi(p,y)=0\}$ is a smooth 
hypersurface through $p$ tangent to $H_pX_G$.
A smooth curve $y_\varepsilon$ in this hypersurface with $y_0=p$
and nonzero tangent at $p$ satisfies
\begin{equation}\label{eq:phase-path}
 \phi(p,y_\varepsilon)=i\lambda\varepsilon^2+O(\varepsilon^3),
                   \qquad \lambda>0.
\end{equation}
Set $q_\varepsilon=-i\phi(p,y_\varepsilon)$.  For small positive
$\varepsilon$,
$\operatorname{Re}q_\varepsilon\ge c\varepsilon^2$ and
$q_\varepsilon/\varepsilon^2\to\lambda$.  Thus the integral in
\eqref{eq:szego-parametrix} is absolutely convergent at these
off-diagonal points.  The exact Laplace identity gives
\[
 \int_0^\infty e^{-q_\varepsilon t}\chi(t)t^q\,dt
       =q!\,q_\varepsilon^{-q-1}+O(1).
\]
Indeed, the bound on $r$ in \eqref{eq:szego-leading} gives
\[
 \left|\int_0^\infty e^{-q_\varepsilon t}\chi(t)
           r(p,y_\varepsilon,t)\,dt\right|
 \le C\int_1^\infty e^{-c\varepsilon^2t}t^{q-1}\,dt
          =O(\varepsilon^{-2q}).
\]
Consequently
\[
 S_{X_G}(p,y_\varepsilon)
 =q!a_0(p,y_\varepsilon)q_\varepsilon^{-q-1}
                         +O(\varepsilon^{-2q})+O(1),
\]
where the $O(1)$ term includes the low-frequency cutoff correction
and the smooth kernel $R$.  
Since $q_\varepsilon/\varepsilon^2\to\lambda$
and $a_0(p,y_\varepsilon)\to a_0(p,p)\ne0$, the leading term 
has order $\varepsilon^{-2q-2}$.  After multiplication by
$\varepsilon^{2q+2}$, the contribution of $r$ is $O(\varepsilon^2)$
and the bounded terms vanish.  Thus
\begin{equation}\label{eq:szego-blowup}
 \varepsilon^{2q+2}S_{X_G}(p,y_\varepsilon)
       \longrightarrow q!a_0(p,p)\lambda^{-q-1}\ne0.
\end{equation}
This contradicts \eqref{eq:offdiag-evaluation}.
Therefore every point has a conductor element that does not vanish
there.  Their nonvanishing sets cover the compact space $X_G$; a
finite subcover gives the claimed family $c_1,\ldots,c_N$.
\end{proof}

\subsection{Conductor nonvanishing via global CR peak functions}
\label{subsec:peaks}
There is a second way to contradict the smooth representation of
evaluation.  Powers of a global CR peak function at $p$ have a value
of one at $p$, while their $L^2$ norms tend to zero.  We construct such
functions using the filling-theoretic arguments below; this
proof is therefore an alternative to the Szeg\H{o}-kernel argument.
A smooth global CR peak function at $p\in X_G$ is a function
$f_p\in\B$ satisfying
\begin{equation}\label{eq:peak}
 f_p(p)=1,\qquad |f_p(x)|<1\quad(x\ne p).
\end{equation}

\begin{proposition}[Peak functions on the reduction]\label{prop:peaks}
Every point of a compact, globally CR embeddable, strictly pseudoconvex
CR manifold of real dimension at least three has a smooth global CR
peak function.
\end{proposition}
\begin{proof}
The construction uses an abstract filling with possibly singular
interior.  We apply the amended filling theorem
\cite[Theorem~10.4$'$]{HLadd} to an embedded connected component
containing $p$.  It gives an abstract compact filling $\overline D$
whose interior $D$ is a normal Stein space with finitely many 
singular points, and whose boundary is smooth and CR identified
 with that component.

We apply the collar-extension theorem
\cite[Theorem~3.1]{HillNacinovich} to the smooth part of the filling.
Choose an inner smooth collar of $\partial D$ 
whose closure is disjoint
from the finite singular set.  Glue the collar extension supplied by
\cite[Theorem~3.1]{HillNacinovich} to the original filling along this
inner collar, on which the complex structures agree.  This gives a
complex space $W$ containing $\overline D$ as a compact domain with an
embedded smooth strictly pseudoconvex boundary.
%
Since $D$ is Stein and $\partial D$ is a smooth strictly
pseudoconvex hypersurface in the regular part of $W$, $\overline D$
has a fundamental system of Stein neighborhoods in $W$.  For a Stein
manifold this is \cite[Theorem~4.1]{HillNacinovich}.  Its proof bumps
$\partial D$ outward inside the regular part of $W$ and uses the
Andreotti--Grauert restriction isomorphisms
$H^q(\mathcal N,\mathcal F)\cong H^q(D,\mathcal F)$, $q\ge1$, for
coherent analytic sheaves $\mathcal F$ on the bumped neighborhoods
$\mathcal N$, which hold on complex spaces \cite{AG62}.  So it applies
unchanged here (see also \cite[Theorem~1.1]{Starcic}, with its
auxiliary compact set empty).

Near $p$ the ambient space is a complex manifold.  
A holomorphic support
function $q$ can be chosen with $q(p)=0$ and
$\operatorname{Re}q>0$ on the closed filling near $p$, apart from $p$ itself.  
For example, take the negative Levi polynomial of a locally strictly
plurisubharmonic defining function, after shrinking the coordinate
neighborhood.  Choose a real cutoff $0\le\chi\le1$ supported there
and equal to one near $p$.  Its support avoids the interior singular
points.  The compact set
$\supp(d\chi)\cap\{q=0\}$ is disjoint from $\overline D$.
Choose a Stein neighborhood $\mathcal N$ of $\overline D$ missing
this set.  Then
\[
 \alpha=\frac{\bar\partial\chi}{q},
\]
extended by zero off the coordinate neighborhood, is a smooth
$\bar\partial$-closed $(0,1)$-form on $\mathcal N$, identically
zero near $p$ and near the singular points.

\handrev{To solve $\bar\partial v=\alpha$ on $\mathcal N$, 
use smooth local
primitives on regular coordinate sets and the zero primitive near the
singular points.}  Their differences on overlaps
are holomorphic.  Cartan's Theorem~B for $\OO_{\mathcal N}$ eliminates
the additive Cousin cocycle, giving $v$ smooth near $\overline D$ and
holomorphic near $p$ and the singular points.  In particular $v$ is
bounded on $\overline D$.

The function $H=\chi/q-v$ is holomorphic on $D$ and smooth on
$\overline D\setminus\{p\}$.  Since
$\Re(\chi/q)\ge0$ where it is used, choose a real constant
$\kappa$ so that $\Re(H+\kappa)>0$ on
$\overline D\setminus\{p\}$.  Set
\[
 f_p=\frac{H+\kappa}{H+\kappa+1}.
\]
It has modulus strictly less than one away from $p$, since
$|a+1|^2-|a|^2=2\operatorname{Re}a+1>0$ for $a=H+\kappa$.
Near $p$,
\begin{equation}\label{eq:peak-smooth-at-p}
 1-f_p=\frac{q}{1+(\kappa+1-v)q}.
\end{equation}
The denominator is nonzero near $p$.  Thus $f_p$ extends smoothly,
in fact holomorphically in the local ambient chart, at $p$ with value
one.  Its boundary value is the required CR peak function.  Set it to
zero on any other connected components.  \handrev{The preceding
construction proves the proposition directly.}
\end{proof} 

Proposition~\ref{prop:peaks} can also be established by a direct, 
microlocal construction relying on the Szeg\H{o} kernel 
asymptotic expansion, bypassing the complex filling, 
collar extension, and Stein neighborhood theorems. 
Specifically, one can adapt the method used in 
the proof of \cite[Theorem~1.10]{HM}.
\begin{proof}[Alternative proof of Proposition~\ref{prop:peaks} 
via Toeplitz asymptotics]
Fix $p \in X_G$. Let $dv_{X_G}$ be the volume form 
induced by the Levi form on $X_G$. 
It is not difficult to see that there exists an elliptic first-order 
self-adjoint pseudodifferential operator $P \in L_{cl}^1(X_G)$ 
such that the function $\sigma^0_P(\omega_{0,X_G}(x))$ 
attains a strict global minimum at $p$, where $\sigma^0_P$ 
denotes the principal symbol of $P$. Moreover, we take $P$ so that 
\begin{equation}\label{e-gue261003yyd}
\begin{split}
&d\Bigr(\frac{1}{\sigma^0_P(\omega_{0,X_G}(x))}\Bigr)\Big|_{x=p}=0,\\
&\mbox{${\rm Hess\,}\Bigr(\frac{1}{\sigma^0_P(\omega_{0,X_G}(x))}\Bigr)\Big|_{x=p}$ 
is negative definite}.
\end{split}
\end{equation}

Consider the associated elliptic self-adjoint Toeplitz operator 
$T_P = S_{X_G} P S_{X_G}$ on $L^2(X_G)$, 
where $S_{X_G}: L^2(X_G) \to H_b^0(X_G)$ is the Szeg\H{o} projector. 
Let $0 < \lambda_1 \le \lambda_2 \le \dots \to +\infty$ 
be the sequence of positive eigenvalues of $T_P$, and let 
$\{f_j\}_{j=1}^\infty \subset \mathcal{B} = 
\mathscr{C}^\infty(X_G) \cap H_b^0(X_G)$ 
be a corresponding orthonormal system of smooth CR eigenfunctions.

Fix a non-negative cut-off function 
$\eta \in \mathscr{C}_c^\infty((0, \infty))$ with $\eta \not\equiv 0$. 
For $k > 0$, consider the localized spectral kernel
\[
\eta_k(T_P)(x, y) = \sum_{j=1}^\infty \eta(k^{-1}\lambda_j) 
f_j(x)\overline{f_j(y)} \in \mathscr{C}^\infty(X_G \times X_G).
\]
By the diagonal asymptotic expansion \cite[Corollary~1.2]{HHMS}, as $k \to +\infty$,
\[
\eta_k(T_P)(x, x) = k^{q+1} \mathcal{A}_0(x) + O(k^q), \quad q = n-d,
\]
where the leading coefficient is given by
\[
\mathcal{A}_0(x) = \frac{1}{2\pi^{q+1}} 
\frac{1}{\sigma^0_P(\omega_{0,X_G}(x))^{q+1}} 
\int_0^\infty \eta(t) t^q \, dt = 
\frac{C_\eta}{\sigma^0_P(\omega_{0,X_G}(x))^{q+1}}, \quad C_\eta > 0.
\]
We work on local coordinates $x=(x_1,\ldots,x_{2q+1})$ 
defined near $p$ with $x(p)=0$. Near $p$, consider Taylor expansion 
of $\eta_k(T_P)(x, x)$ at $x=p$: From \eqref{e-gue261003yyd}, we have 
\begin{equation}\label{e-gue261004yydI}\eta_k(T_P)(x,x)-\eta_k(T_P)(p,p)
\leq -ck^{q+1}|x-p|^2+O(k^{q}).\end{equation}
Fix $0<\varepsilon\ll1$. From \eqref{e-gue261004yydI}, 
we see that if $\varepsilon>0$ is small enough, then 
\begin{equation}\label{e-gue261004yydII}
\eta_k(T_P)(x,x)<\eta_k(p,p),\ \ \mbox{for all $|x-p|\leq 
k^{-\frac{1}{2}+\varepsilon}$, $x\neq p$}.
\end{equation}
Form the off-diagonal expansion in~\cite[Theorem 1.1]{HHMS}, 
it is straightforward to see that 
\begin{equation}\label{e-gue261004ycdr}
\eta_k(T_P)(x,p)=O(k^{-\infty})\ \ \mbox{on $|x-p|>k^{-\frac{1}{2}+\varepsilon}$}.
\end{equation}

Fix a sufficiently large $k$, and define the global smooth CR function
\[
f_p(x) := \frac{\eta_k(T_P)(x, p)}{\eta_k(T_P)(p, p)} = 
\frac{\sum_{j=1}^\infty \eta(k^{-1}\lambda_j) 
\overline{f_j(p)} f_j(x)}{\sum_{j=1}^\infty 
\eta(k^{-1}\lambda_j) |f_j(p)|^2} \in \mathcal{B}.
\]
From Cauchy inequalities, we have 
\begin{equation}\label{e-gue261004ycdm}
|\eta_k(T_P)(x,p)|^2\leq|\eta_k(T_P)(x,x)||\eta_k(T_P)(p,p)|.
\end{equation}
From \eqref{e-gue261004yydII}, \eqref{e-gue261004ycdr}
and \eqref{e-gue261004ycdm}, we get $f_p(p) = 1$
and $|f_p(x)| < 1$ for all $x\in X_G$, $x\neq p$. 
Thus, $f_p$ is the desired smooth global CR peak function.
\end{proof}
\begin{proof}[Second proof of Lemma~\ref{lem:nocommon}]
Suppose every $c\in\I$ vanishes at a common point
$p\in X_G$.  Then
\eqref{eq:evaluation-representative} supplies a smooth $h$
representing evaluation at $p$.  Choose $f_p$ from
Proposition~\ref{prop:peaks}. The powers $f_p^\nu$ belong to $\B$,
are bounded by one, and tend to zero away from $p$. A point has zero
smooth volume, so dominated convergence gives
$\norm{f_p^\nu}_{X_G,0}\to0$.  But
\[
 1=f_p^\nu(p)=(f_p^\nu,h)_{X_G},\qquad
 |(f_p^\nu,h)_{X_G}|
 \le\norm{f_p^\nu}_{X_G,0}\norm h_{X_G,0}\longrightarrow0,
\]
a contradiction.  Thus no point is a common zero of the conductor;
compactness gives a finite nonvanishing family as above.
\end{proof}

\subsection{Local complex quotients and BRT coordinates}
\label{sec:local}
\leavevmode\structrev{We recall BRT coordinates, which give a local
complex quotient on which invariant CR functions become holomorphic.
Away from $Y$, a suitable infinitesimal group direction is locally
transverse to the CR distribution, even when the full $G$-orbit space
is not smooth.}

\begin{lemma}[Local BRT normal form]\label{lem:BRT}
Let $V_0$ be a smooth real vector field on a neighborhood $U$
of $x_0$ in a CR manifold of hypersurface type.  Suppose
\[
 TX=HX\oplus\RR V_0,\qquad
 [V_0,{\mathscr C}^\infty(U,T^{1,0}X)]\subset
 {\mathscr C}^\infty(U,T^{1,0}X).
\]
There are coordinates $(z,\theta)$ near $x_0$, with
$z\in V\subset\CC^n$ and $\theta\in(-\varepsilon,\varepsilon)$,
in which
\begin{equation}\label{eq:BRT}
 V_0=\partial_\theta,\qquad
 Z_j=\partial_{z_j}+i\varphi_{z_j}(z,\bar z)\partial_\theta,
       \quad 1\le j\le n,
\end{equation}
form a frame of $T^{1,0}X$.  Here $\varphi$ is smooth and 
real-valued, independent of $\theta$.
\end{lemma}
\handrev{These are the BRT coordinates \cite{BRT}, in the sign 
convention of \cite[Theorem~3.2, (3.14)--(3.17)]{HHtorsion}; 
the proof below establishes the local version needed here.}

\begin{proof}
A flow box identifies a neighborhood with $W\times I$, where $W$ is
initially a real $2n$-manifold, $V_0=\partial_t$, and
$p_0(y,t)=y$.  The restriction $dp_0|_{HX}$ is a smooth real bundle
isomorphism.  Denote its inverse on each fiber by $\ell_x$.
\handrev{Write $\Phi_\tau=\Phi^{V_0}_\tau$ for the local flow of
$V_0$.}  Then
\[
 p_0\Phi_\tau=p_0,\quad d\Phi_\tau(HX)=HX,\quad
 d\Phi_\tau J=Jd\Phi_\tau,\quad
 \ell_{\Phi_\tau(x)}=d\Phi_\tau\ell_x.
\]
The last identity follows from uniqueness of the horizontal lift.
Consequently
\[
 J_Wv=dp_{0,x}J_x\ell_x(v),\qquad p_0(x)=y,
\]
is independent of the point $x$ on the local fiber.  It is smooth and
squares to $-I$.

Let $W_1,W_2$ be $(1,0)$ fields for $J_W$, and lift them horizontally
to $V_1,V_2\in T^{1,0}X$.  They are $p_0$-related to $W_1,W_2$.
For every smooth function $f$ on $W$,
\[
 [V_1,V_2](f\circ p_0)=([W_1,W_2]f)\circ p_0.
\]
Thus $dp_0[V_1,V_2]=[W_1,W_2]\circ p_0$.  
The CR involutivity upstairs
proves involutivity of $T^{1,0}W$.  The smooth Newlander--Nirenberg
theorem \cite{NN} supplies local complex coordinates $z$ on $W$.

The unique horizontal lifts of $\partial_{z_j}$ now have the form
\[
 L_j=\partial_{z_j}+\beta_j(z,\bar z)\partial_t.
\]
Their coefficients are independent of $t$, since the flow preserves
horizontal lifts.  Their brackets are both of type $(1,0)$ and
multiples of the transverse vector $\partial_t$:
\[
 [L_j,L_k]
 =\bigl(\partial_{z_j}\beta_k-\partial_{z_k}\beta_j\bigr)\partial_t.
\]
Transversality forces these brackets to vanish.  Hence
$\beta=\sum_j\beta_jdz_j$ is $\partial$-closed.  The local Dolbeault
lemma \cite[Corollary~B.4.3]{MaMarinescu}, or its complex conjugate,
gives a smooth
complex-valued function $h$ with $\partial h=\beta$.  Write
$h=a+i\varphi$ with $a,\varphi$ real and replace $t$ by
$\theta=t-a(z,\bar z)$.  Since $\partial_t=\partial_\theta$ and
$L_j\theta=\beta_j-\partial_{z_j}a=i\varphi_{z_j}$, this gives
\eqref{eq:BRT}.  After shrinking once more, 
the coordinate domain is a
product $V\times(-\varepsilon,\varepsilon)$ with $V$ connected.
\end{proof}

\begin{lemma}[Local complex quotient]\label{lem:quotient}
For every $x_0\in X\setminus Y$, there are a neighborhood
$U\subset X\setminus Y$, a connected complex domain $V\subset\CC^n$,
and a smooth surjective submersion $p:U\to V$ with connected
one-dimensional fibers such that
\begin{equation}\label{eq:pCR}
 dp(T^{0,1}X|_U)=T^{0,1}V.
\end{equation}
Every $u\in\C$, restricted to $U$, is the pullback of a unique
holomorphic function on $V$.  Conversely, every holomorphic
function on $V$ pulls back to a smooth CR function on $U$.
\end{lemma}
\begin{proof}
Since $\mu(x_0)\ne0$, choose a fixed $\xi\in\mathfrak g$ with
$a(x_0)=\omega_0(\xi_X)(x_0)\ne0$.  Set $V_\xi=\xi_X$.
The field is transverse
on the open set $\{a\ne0\}$, and its flow
consists of CR transformations.  
\handrev{Because $HX=\ker\omega_0$ and $a(x_0)\ne0$, we have
$TX=HX\oplus\RR V_\xi$ near $x_0$.} 
Lemma~\ref{lem:BRT} applies to
$V_\xi$, giving \eqref{eq:BRT}.  Define $p(z,\theta)=z$.
Then $\ker dp=\RR V_\xi$, and complex conjugation gives
\begin{equation}\label{eq:BRT-bar}
 \overline Z_j=\partial_{\bar z_j}
                  -i\varphi_{\bar z_j}\partial_\theta,
 \qquad dp(\overline Z_j)=\partial_{\bar z_j}.
\end{equation}
This proves \eqref{eq:pCR}.  If $u\in\C$, its invariance implies
$V_\xi u=0$, so $u(z,\theta)=\widetilde u(z)$ on the connected
flow fibers.  The CR equations become
\[
 0=\overline Z_ju
   =\partial_{\bar z_j}\widetilde u.
\]
Thus $\widetilde u\in\OO(V)$.  Surjectivity of $p$ gives uniqueness,
and \eqref{eq:BRT-bar} proves the converse.
\end{proof}

For the field $V_\xi$ chosen in the proof of
Lemma~\ref{lem:quotient},
\begin{equation}\label{eq:rigid-CR}
 p^*\OO(V)=\{u\in \handrev{\mathscr C}^\infty(U):V_\xi u=0,\ 
 \dbarb u=0\}.
\end{equation}
 There is only an inclusion $\C|_U\subset p^*\OO(V)$, not generally
 equality: invariance under one local flow does not imply full
 $G$-invariance or global extendability.  
 The slice in a flow box is not assumed tangent to $HX$; 
 the complex structure is transported through
 $dp_0|_{HX}$, not obtained by restricting $J$ to a slice.
 For a fixed $x_0$ all the invariant CR functions used later 
 descend on
 the \emph{same} domain $V$.  These observations apply without any
 Levi positivity outside $Y$ and without a global quotient by $G$.
\structrev{In Section~\ref{sec:surjectivity}, we combine conductor
denominators with these local complex quotients to construct smooth
lifts of reduced CR functions.}

\begin{remark}[Invariant CR distributions away from $Y$]
\label{rem:invariant-away}
Every $G$-invariant distribution $u\in\mathcal D'(X)$ satisfying
$\dbarb u=0$ is smooth on $X\setminus Y$.
Indeed, in the product coordinates $U\simeq V\times
(-\varepsilon,\varepsilon)$ of Lemma~\ref{lem:quotient}, invariance
gives $\partial_\theta u=0$ distributionally.  Hence
$u|_U=\widetilde u\otimes1$ for a distribution
$\widetilde u\in\mathcal D'(V)$, and \eqref{eq:BRT-bar} reduces
the CR equations to $\partial_{\bar z_j}\widetilde u=0$ for every
$j$.  Elliptic regularity makes $\widetilde u$ smooth and holomorphic.
This argument uses only the transverse infinitesimal group direction
available where $\mu\ne0$ and the CR invariance of its flow; it
requires neither Levi positivity nor closed range.
\end{remark}

\subsection{Smooth surjectivity of $\sigma_G$}
\label{sec:surjectivity}

Finite codimension first imposes an algebraic constraint on each
$b\in\B$.  The resulting monic relation will later bound the local
quotients used to construct its lift.

\begin{lemma}\label{lem:integrality}
Every $b\in\B$ satisfies a monic polynomial with coefficients in
$\A$, of degree at most $r+1$:
\begin{equation}\label{eq:monic}
 b^\ell+\sum_{k=0}^{\ell-1}a_kb^k=0,
 \qquad a_k\in\A,\quad 1\le\ell\le r+1.
\end{equation}
\end{lemma}
\begin{proof}
In the $r$-dimensional complex vector space $\B/\A$, the $r+1$
vectors $[b],\ldots,[b^{r+1}]$ are dependent.  Take a nontrivial
relation and let $\ell$ be its highest index with nonzero coefficient.
Normalize that coefficient to one.  Then
$b^\ell+\sum_{k=1}^{\ell-1}\gamma_kb^k=a\in\A$, with
$\gamma_k\in\CC$.  Since $\CC1\subset\A$, set $a_0=-a$ and
$a_k=\gamma_k1$ for $k\ge1$.  This proves \eqref{eq:monic}.
\end{proof}

Choose $c_1,\ldots,c_N\in\I$ with no common zero as in
Lemma~\ref{lem:nocommon}.  Fix $b\in\B$ and a relation
\eqref{eq:monic}.  Since $c_j$, $c_jb$, and $a_k$ belong to $\A$,
there are unique lifts
\begin{equation}\label{eq:lifts}
 C_j=\sigma_G^{-1}(c_j),\qquad
 B_j=\sigma_G^{-1}(c_jb),\qquad
 A_k=\sigma_G^{-1}(a_k),
\end{equation}
Multiplicativity and injectivity of $\sigma_G$ 
(cf.\ Proposition \ref{prop:injective}) give the
identities
\begin{equation}\label{eq:cross}
\begin{split}
 &C_iB_j=C_jB_i,\\
 &B_j^\ell+\sum_{k=0}^{\ell-1}A_kB_j^kC_j^{\ell-k}=0
                        \quad\text{on }X.
\end{split}
\end{equation}
For the second identity apply $\sigma_G$ and obtain $c_j^\ell$
times \eqref{eq:monic}.
The set
\begin{equation}\label{eq:K}
 K=\{x\in X:C_1(x)=\cdots=C_N(x)=0\}
\end{equation}
is closed and $G$-invariant.  The restriction formulas
$C_j|_Y=\pi^*c_j$ show that $K\cap Y=\varnothing$.  Equation
\eqref{eq:cross} permits the definition
\begin{equation}\label{eq:quotient-u}
 u=B_j/C_j\quad\text{where }C_j\ne0.
\end{equation}
It is well defined, smooth, invariant, and CR on $X\setminus K$.
The quotient rule proves the CR condition, and
$B_j|_Y=\pi^*(c_jb)$ shows $u|_Y=\pi^*b$.

\begin{lemma}[Empty interior]\label{lem:empty}
$\Int_XK=\varnothing$.
\end{lemma}
\begin{proof}
Set $O=\Int_XK$ and take $x\in\overline{O}$, the closure of $O$. 
Since $K$ is closed,
\begin{equation}\label{eq:closure-in-K}
 \overline{O}\subset K\subset X\setminus Y.
\end{equation}
Use Lemma~\ref{lem:quotient} at $x$ to obtain $p:U\to V$ with
connected $V$.  All the $C_j$ descend to holomorphic functions
$\widetilde C_j$ on $V$.  As $x\in\overline{O}$, the set $U\cap O$
is nonempty and open.  A smooth submersion is open, so $p(U\cap O)$
is a nonempty open subset of $V$.  Every $\widetilde C_j$ vanishes
there, hence on all of $V$ by the holomorphic identity theorem.  Thus
$U\subset K$, so $x\in O$.  Hence $\overline{O}\subset O$ and $O$ is
closed in $X$.
Connectedness of $X$, nonemptiness of $Y$, and $K\cap Y=\varnothing$
exclude $O=X$.  Therefore $O=\varnothing$.
\end{proof}

\begin{proposition}[Smooth surjectivity]\label{prop:surjective}
The canonical map $\sigma_G:\C\to\B$ is surjective.
\end{proposition}
\begin{proof}
It remains to extend \eqref{eq:quotient-u} across $K$.
On each local complex quotient, we show that its meromorphic 
representative satisfies a monic polynomial relation. 
This gives a local bound, allowing us to apply the removable-singularity theorem.
For $x\in K$,
Lemma~\ref{lem:quotient} gives a neighborhood $U$ and a connected
complex domain $p:U\to V$.  The functions in \eqref{eq:lifts}
descend to holomorphic functions
$\widetilde C_j,\widetilde B_j,\widetilde A_k$ on this same $V$.
Since $K$ has empty interior, some $\widetilde C_j$ is not 
identically zero.  Set $E=\{\widetilde C_j=0\}$ and
\[
 F_0=\widetilde B_j/\widetilde C_j\quad\text{on }V\setminus E.
\]
The set $E$ is a proper analytic subset.  Surjectivity of $p$ makes
the second identity of \eqref{eq:cross} 
an identity downstairs, so
\begin{equation}\label{eq:local-monic}
 F_0^\ell+\sum_{k=0}^{\ell-1}\widetilde A_kF_0^k=0.
\end{equation}
If $W\Subset V$ and $M\ge1$ bounds all the coefficients there, a
root of this equation satisfies $|F_0|\le1+M$.  Indeed, for
$R=|F_0|>1+M$ one would have
\[
 R^\ell\le M\sum_{k=0}^{\ell-1}R^k
 =M\frac{R^\ell-1}{R-1}<R^\ell.
\]
The bounds hold on $W\setminus E$, including up to $E$.  Bounded
holomorphic removability \cite[Lemma~2.3.22]{MaMarinescu} gives 
a unique $\widehat F_0\in\OO(V)$ extending $F_0$.  

For every $i$, the holomorphic function
$\widetilde C_i\widehat F_0-\widetilde B_i$ vanishes on
$V\setminus E$ by \eqref{eq:cross}, and hence on all of $V$.
Consequently the smooth CR pullback $u_U=p^*\widehat F_0$ satisfies
\begin{equation}\label{eq:all-numerators}
 C_i u_U=B_i\quad\text{on }U\quad\text{for every }i.
\end{equation}
\handrev{At any point of $U\setminus K$, 
at least one $C_i$ is nonzero; using
that index in \eqref{eq:all-numerators} proves $u_U=u$ throughout
$U\setminus K$.}

Repeat the construction near every point of $K$.  On an overlap $W$,
two local extensions agree on $W\setminus K$, which is dense by
Lemma~\ref{lem:empty}.  Continuity gives agreement on all of $W$.
They glue with $u$ to a global smooth CR function $\widehat u$.

Finally, for fixed $g\in G$, $g^*\widehat u-\widehat u$ is
continuous and zero on the dense invariant set $X\setminus K$.
It vanishes everywhere.  Thus $\widehat u\in\C$.
Since $Y\cap K=\varnothing$,
$\widehat u|_Y=\pi^*b$, so $\sigma_G\widehat u=b$.
\end{proof}

\subsection{Proof of Theorem~\ref{thm:main}: 
Completion on the Sobolev scale}\label{s:proofSob}
\handrev{In this subsection, we complete the proof of
Theorem~\ref{thm:main}.  We combine the surjectivity on smooth CR
functions proved above with the Fredholm statement in
Theorem~\ref{thm:HMM-input} to show that the
canonical map is a bounded isomorphism at every real Sobolev order,
with the shift $d/4$.  The resulting estimates also give continuity
of the inverse in the smooth Fr\'echet topology.  Finally, we use the
algebra isomorphism on smooth CR functions and the connectedness of
$X$ to prove that the reduction $X_G$ is connected.}
\begin{proof}[Proof of Theorem~\ref{thm:main}]
The range of $T_0=\sigma_{G,d/4}$ is closed by
Theorem~\ref{thm:HMM-input} and contains $\B$ by
Proposition~\ref{prop:surjective}.  
To verify density explicitly, let
$b\in\Hb(X_G)$ and take $f_\nu\in \handrev{\mathscr C}^\infty(X_G)$ 
with $f_\nu\to b$ in $L^2$.  Smoothness preservation and 
orthogonality of $S_{X_G}$ give
\[
 b_\nu=S_{X_G}f_\nu\in\B,\qquad
 \norm{b_\nu-b}_{X_G,0}
 \le\norm{f_\nu-b}_{X_G,0}\longrightarrow0.
\]
Hence $T_0$ is surjective.  It is injective by
Proposition~\ref{prop:injective}, so its index is zero.
The index is independent of $s$, and the kernel vanishes at every
order.  Thus $\dim\coker\sigma_{G,s}=0$ for every real $s$.
Fredholmness includes closed range, so each map is 
a bounded bijection
between Hilbert spaces.  The bounded inverse theorem gives
\eqref{eq:main-estimate}.

Compatibility \eqref{eq:compatibility} and injectivity 
of Proposition \ref{prop:injective} also identify the
inverse lifts at different orders.  Applied to a smooth target, 
the inverses have arbitrarily high Sobolev regularity 
and coincide with its smooth lift.  
The Sobolev estimates for arbitrarily high orders show
continuity of the smooth inverse in the Fr\'echet topology.  
The forward map is continuous on smooth functions 
by restriction and descent, or by its Sobolev estimates.  
Multiplicativity and bijectivity give the stated
Fr\'echet algebra isomorphism.

To prove connectedness, if $X_G$ were disconnected, 
choose the indicator
$b$ of one nonempty proper union of components.  It is a smooth CR
function satisfying $b^2=b$, and it is nonconstant.  Its lift $u\in\C$
satisfies $\sigma_G(u^2-u)=0$.  Injectivity yields $u^2=u$, so $u$
takes values in $\{0,1\}$.  Connectedness of $X$ implies that $u$ 
is constant, and then $b=\sigma_Gu$ is constant, a contradiction.
\handrev{This completes the proof of Theorem~\ref{thm:main}.}
\end{proof}

\begin{remark}[The finite-group case]\label{rem:finite}
If $\dim X=3$ and $\dim X_G\ge3$, \eqref{eq:dimension} gives $d=0$.
A compact zero-dimensional Lie group is finite, $Y=X$, and freeness
makes $\pi:X\to X/G$ a finite CR covering.  Invariant smooth or
Sobolev CR functions descend uniquely through this covering, with
Sobolev shift zero.  
This proves the assertion also when \(n=1\).
\end{remark}

\section{Sasakian and holomorphic 
quantization}\label{sec:applications}
The CR isomorphism has two quantization consequences. 
 Its equivariance under the Reeb torus gives the Sasakian result 
 in every weight, while
the circle-bundle correspondence relates it to canonical maps on
holomorphic sections.  

This section is organized as follows. 
In Section \ref{subsec:sasaki-proof},
we prove Corollary~\ref{cor:sasaki} as a consequence of Theorem 
\ref{thm:main}. In Section \ref{subsec:circle}, we explain the 
holomorphic reduction and the circle-bundle correspondence.
In Section \ref{subsec:powers},
\negrev{we prove Theorem~\ref{thm:line} independently of
Theorem~\ref{thm:main}.  A maximum-norm argument gives negative-degree
vanishing, while powers and holomorphic removability give the
nonnegative-degree isomorphisms, including zero-dimensional reduced
bases by stabilization.
In Section \ref{s4.4}, we compare with classical holomorphic and index results.}

\subsection{Proof of Corollary~\ref{cor:sasaki}: 
Sasakian weight spaces}\label{subsec:sasaki-proof}
\handrev{To obtain the weight-space isomorphism, we apply
Theorem~\ref{thm:main} to the Sasakian reduction and use the
equivariance of the canonical map under the Reeb torus.  Uniqueness
then ensures that a reduced CR function of weight $\alpha$ lifts to
an invariant CR function of the same weight.}

\begin{proof}[Proof of Corollary~\ref{cor:sasaki}]
\noterev{The identities $\omega_0(R)=1$ and
$\iota_Rd\omega_0=0$ uniquely determine the Reeb vector field, so
preservation of $\omega_0$ by $G$ implies preservation of $R$.
Consequently, $G$ commutes with the Reeb flow and its closure $\Tt$.
Since $\Tt$ preserves $\omega_0$ and the CR structure, for
$t\in\Tt$ and $\xi\in\mathfrak g$ we have}
\[
 \langle\mu(t\cdot x),\xi\rangle
 =\omega_{0,t\cdot x}\bigl(dt_x\xi_X(x)\bigr)
 =\langle\mu(x),\xi\rangle.
\]
Thus $\mu$ is $\Tt$-invariant and $\Tt$ preserves $Y=\mu^{-1}(0)$.
Infinitesimally, Cartan's formula and
$\mathcal L_{\xi_X}\omega_0=0$ give
\[
 d\langle\mu,\xi\rangle(R)=d\omega_0(R,\xi_X)=0,
\]
where the last equality is the Reeb identity
$\iota_Rd\omega_0=0$.  Hence $R$ is tangent to $Y$.
The torus descends to a CR action on $X_G$, and $R$ descends to
$\widetilde R$.  Pulling back under $\pi$ and using
$\pi^*\omega_{0,G}=\iota^*\omega_0$ gives
\[
 \omega_{0,G}(\widetilde R)=1,\qquad
 \iota_{\widetilde R}d\omega_{0,G}=0.
\]
Since $\widetilde R$ also preserves the reduced CR structure, $X_G$
is Sasakian; compare \cite[Remark~2.3 and the discussion preceding
(1.9)]{HMM}.

\handrev{By the embedding theorem of Marinescu--Yeganefar
\cite[Theorem~5.2]{MY07}, compact Sasakian manifolds of 
real dimension
at least three are CR embeddable.  Apply this to $X$ and to each
component of $X_G$, including when $\dim_{\RR}X_G=3$.
The equivalence between embeddability and closed-range recalled in
Section~\ref{sec:prelim} then gives closed range for their tangential
CR operators on functions.  Thus Theorem~\ref{thm:main} applies.\par}

\handrev{We also verify the regularity and finite-dimensionality 
of the weight spaces defined before Corollary~\ref{cor:sasaki}.
Let $\rho:\RR\to\Tt$ be the
\noterev{one-parameter subgroup defined by
$\rho(s)\cdot x=\Phi^R_{-s}(x)$, where $\Phi^R_s$ is the Reeb flow
at time $s$.  This choice agrees with the infinitesimal-action
convention in \eqref{eq:moment-convention}.  Define}
$\lambda_\alpha\in\RR$ by
$\alpha(\rho(s))=e^{i\lambda_\alpha s}$.
For $Z=X,X_G$, put $R_X=R$ and $R_{X_G}=\widetilde R$.
Define the differential operator
\[
 P_Z=\bar\partial_{b,Z}^{\,*}\bar\partial_{b,Z}-R_Z^2,
\]
where the adjoint is taken with respect to the Sasakian metric.
\noterev{For $u\in H^0_{b,\alpha}(Z)$, the convention above gives
$\rho(s)^*u=u\circ\Phi^{R_Z}_s=e^{i\lambda_\alpha s}u$.
Differentiating yields $R_Zu=i\lambda_\alpha u$}
distributionally.  Since $\bar\partial_{b,Z}u=0$, it follows that
\[
 P_Zu=\lambda_\alpha^2u.
\]
The operator $P_Z$ is elliptic: the principal symbol of the first
term controls horizontal covectors, while $-R_Z^2$ controls the
remaining direction in $TZ=HZ\oplus\RR R_Z$.
Elliptic regularity and compactness imply that its eigenspaces are
smooth and finite-dimensional, and hence so are the character spaces;
see also \cite[the discussion following (1.9)]{HMM}.\par}

\handrev{Canonical restriction intertwines the torus actions 
and the Reeb fields:}
\[
 \sigma_G R=\widetilde R\sigma_G
\]
on smooth invariant CR functions, as in \cite[(1.10)]{HMM}.
For $b\in H^0_{b,\alpha}(X_G)$, let $u\in\C$ be its unique
smooth lift.  For $t\in\Tt$, equivariance gives
$\sigma_G(t^*u-\alpha(t)u)=0$, so injectivity gives
$t^*u=\alpha(t)u$.  Therefore $u\in H^0_{b,\alpha}(X)^G$.
Injectivity on this space follows from global injectivity.
\end{proof}

\subsection{Holomorphic reduction and the circle-bundle 
correspondence}\label{subsec:circle}
\handrev{The unit circle bundle $X=S(L^*)$ identifies holomorphic
sections with fixed-weight CR functions and matches their canonical
restriction maps.}
The hypotheses and notation of Theorem~\ref{thm:line} remain in force;
\handrev{in particular, $P=\widetilde\mu^{-1}(0)$.}
The infinitesimal action in \eqref{eq:Kostant} is
\[
 \mathcal L_\xi s=\left.\frac{d}{dt}\right|_{0}
       \bigl(\exp(t\xi)\cdot s\bigr),\qquad
 (g\cdot s)(x)=g\bigl(s(g^{-1}x)\bigr).
\]
As in
\cite[(6.2)]{HMM},
\begin{equation}\label{eq:Kostant-derivative}
 d\langle\widetilde\mu,\xi\rangle
       =\iota_{\xi_M}\omega,\qquad \omega=\frac{i}{2\pi}R^L.
\end{equation}
\handrev{For comparison, invariance of $\omega_0$ and 
Cartan's formula give $d\langle\mu,\xi\rangle
=d(\iota_{\xi_X}\omega_0)=-\iota_{\xi_X}d\omega_0$.}
Near $P$, this is a K\"ahler form.  The local K\"ahler reduction
construction therefore supplies the holomorphic structures on $M_G$
and $L_G$.  Only a neighborhood of $P$ is needed. 
 The reduced curvature
is the restriction of $R^L$ to the reduced complex horizontal
directions; consequently $L_G$ is positive if $\dim_\CC M_G\ge1$.
\structrev{The reduced holomorphic structures on $M_G$ and $L_G$ and
the reduced curvature formula are those described before
\cite[Theorem~1.4]{HMM}.}

Let $p_L:X=S(L^*)\to M$ be the unit circle bundle. 
\handrev{The CR structure is
$T^{1,0}X=\CC TX\cap T^{1,0}L^*$. 
The unit circle $S^1$ acts by fiberwise multiplication 
on $L$, thus on $L^*$ as 
$(v,e^{i\theta})\mapsto ve^{-i\theta}$, 
with globally defined generator
$\partial_{\theta,v}=
\left.\frac{d}{d\theta}\right|_{\theta=0} ve^{i\theta}$ 
via (\ref{eq:moment-convention}).  }
The connection form is normalized as in \cite[(2.24)]{HMM}:
\begin{equation}\label{eq:circle-curvature}
 \omega_0(\partial_\theta)=1,\qquad
 d\omega_0=p_L^*(iR^L).
\end{equation}
The $G$-action on $L^*$ commutes with fiber rotation.  
With the precise
conventions \eqref{eq:moment-convention} and \eqref{eq:Kostant}, its
CR moment map is
\begin{equation}\label{eq:circle-moment}
 \mu=-2\pi\,p_L^*\widetilde\mu.
\end{equation}
The sign follows directly from the conventions.  For a
local section $s$ of $L$, write $\widehat s(v)=v(s(p_L(v)))$.
The definitions give $\xi_X\widehat s=\widehat{\mathcal L_\xi s}$.
On the other hand,
$\xi_X=(\xi_M)^H+\omega_0(\xi_X)\partial_\theta$ and
$\partial_\theta\widehat s=i\widehat s$.  Thus
$\mathcal L_\xi s=\nabla^L_{\xi_M}s+i\omega_0(\xi_X)s$.
Comparison with \eqref{eq:Kostant} yields \eqref{eq:circle-moment}.
In particular,
\begin{equation}\label{eq:circle-reduction}
 Y=p_L^{-1}(P),\qquad X_G=Y/G\simeq S(L_G^*).
\end{equation}
Levi positivity near $Y$ follows from \eqref{eq:circle-curvature};
freeness follows from freeness on $P$.

\handrev{For $m\in\mathbb Z$, evaluation is the canonical
 correspondence}
\begin{equation}\label{eq:section-evaluation}
 \begin{gathered}
 \widehat s(v)=v^{\otimes m}(s(p_L(v))),\\
 H^0(M,L^m)^G\simeq H^0_{b,m}(X)^G,\qquad
 H^0(M_G,L_G^m)\simeq H^0_{b,m}(X_G).
 \end{gathered}
\end{equation}
\handrev{For $m<0$, $v^{\otimes m}$ means
$(v^{-1})^{\otimes(-m)}$, where $v^{-1}\in L_{p_L(v)}$ is the unique
vector satisfying $v(v^{-1})=1$.}
Here $H^0_{b,m}$ denotes smooth CR functions satisfying
$f(v e^{i\theta})=e^{im\theta}f(v)$.  This is the correspondence in
\cite[(2.20)--(2.23)]{HMM}.  For example, in 
a local holomorphic frame
of metric $e^{-2\varphi}$, a scalar section coefficient $s(z)$ gives
$\widehat s(z,\theta)=e^{im\theta-m\varphi(z)}s(z)$.
The local BRT equations reduce exactly to $\bar\partial s=0$
\handrev{for every $m\in\mathbb Z$}.
This observation also shows that a fixed-weight distributional CR
function is a distributional holomorphic section and is therefore
smooth by elliptic regularity.

The canonical restriction in \eqref{eq:section-evaluation} 
is precisely
$A_m$, and it preserves multiplication across tensor degrees:
\begin{equation}\label{eq:graded-multiplication}
 A_{m+k}(s\otimes t)=A_m(s)\otimes A_k(t).
\end{equation}
\handrev{Thus multiplication is encoded in the graded section ring.}

\handrev{Assume now that $\dim_\CC M_G\ge1$.  
With the dual reduced metric, set
\begin{equation}\label{eq:reduced-disk}
 D_G=\{v\in L_G^*:|v|_{h^{L_G^*}}<1\}.
\end{equation}
The projection $p_{L_G}:X_G\to M_G$ is the unit circle bundle
projection, and $X_G=S(L_G^*)=\partial D_G$.  
The Levi form of $\partial D_G$ is $p_{L_G}^*(iR^{L_G})$, so 
positivity of $L_G$
makes $D_G$ a strictly pseudoconvex domain 
in the complex manifold $L_G^*$. 
Since $M_G$ is compact, $\overline{D_G}$ is a compact 
complex manifold
with smooth strictly pseudoconvex boundary $X_G$.  Kohn's theorem
\cite{Kohn86} therefore gives closed range for $\dbarG$ on functions.

For $n\ge2$, Theorem~\ref{thm:main} now applies to $X=S(L^*)$.
\handrev{Fiber rotation commutes with $G$ and preserves $Y$ by
\eqref{eq:circle-moment}.  The canonical map is therefore
$S^1$-equivariant, as is its inverse.  Its restriction to each integer
weight is the section map $A_m$ under \eqref{eq:section-evaluation}.
For $m<0$, positivity of $L_G$ implies
$H^0(M_G,L_G^m)=0$.  Indeed, near a positive maximum of the norm of a
nonzero holomorphic section $s$ one would have
$i\partial\bar\partial\log|s|^2_{h^{L_G^m}}
=-m\,iR^{L_G}>0$, a contradiction.  Hence the source
weight space vanishes as well, proving
\eqref{eq:negative-vanishing}.  The nonnegative weights prove
\eqref{eq:line-all}, \noterev{completing the circle-bundle proof of
Theorem~\ref{thm:line} in the case $n\ge2$ and $\dim_\CC M_G\ge1$.}
\negrev{The following subsection proves
Theorem~\ref{thm:line} in full, including negative-degree source
vanishing for zero-dimensional reduced bases.}}}

\subsection{An independent proof of Theorem \ref{thm:line}
 using large tensor powers}\label{subsec:powers}
\leavevmode\negrev{We prove Theorem~\ref{thm:line} without invoking
Theorem~\ref{thm:main}: negative-degree vanishing follows directly
from the Kostant identity and the maximum principle, and the
nonnegative-degree isomorphisms follow from
\cite[Theorem~1.4]{HMM}, which gives bijectivity for all sufficiently
large tensor powers.}
\begin{proof}[Proof of Theorem~\ref{thm:line}]
\negrev{\emph{Negative degrees.}
Let $m<0$ and suppose that $0\ne s\in H^0(M,L^m)^G$.
Compactness gives a maximum point $x$ of $|s|^2_{h^{L^m}}$, with
$s(x)\ne0$.  On the nonvanishing locus of $s$, put
$\varphi=\log|s|^2_{h^{L^m}}$.  Invariance and the tensor-power
version of \eqref{eq:Kostant} give, for every $\xi\in\mathfrak g$,
\[
 \nabla^{L^m}_{\xi_M}s
   =2\pi i m\langle\widetilde\mu,\xi\rangle s.
\]
Holomorphicity gives
$\nabla^{L^m}_{J\xi_M}s=i\nabla^{L^m}_{\xi_M}s$, and hence
\[
 d\varphi(J\xi_M)=-4\pi m\langle\widetilde\mu,\xi\rangle.
\]
At $x$ all these derivatives vanish.  Since $m\ne0$, it follows that
$\widetilde\mu(x)=0$, so $x\in P$.  The curvature identity near this
nonvanishing point is
\[
 i\partial\bar\partial\varphi=-m\,iR^L>0
 \quad\text{at }x,
\]
contradicting the nonpositive complex Hessian at a local maximum.
Thus $H^0(M,L^m)^G=0$ for every $m<0$, including when $M_G$ is
zero-dimensional.  If $\dim_\CC M_G\ge1$, the same maximum principle
applied to the positive bundle $L_G$ gives
$H^0(M_G,L_G^m)=0$.  This proves \eqref{eq:negative-vanishing} and
bijectivity of $A_m$ in negative degrees in that case.

\emph{Finite groups.}
If $d=0$, then $G$ is finite, $\widetilde\mu\equiv0$, and $P=M$.
The free action makes $M\to M_G$ a finite holomorphic covering, and
holomorphic descent identifies $H^0(M,L^m)^G$ with
$H^0(M_G,L_G^m)$ for every integer $m$.  The quotient $M_G$ is
connected, and positivity near $P=M$ is global.  In particular, this
handles $n=1$ with $\dim_\CC M_G\ge1$, since
$\dim_\CC M_G=n-d$ then forces $d=0$.}

\handrev{\negrev{We next prove the nonnegative-degree 
assertions when
$\dim_\CC M_G\ge1$ and $d\ge1$, so $n\ge2$.}  We invoke only
this positive-dimensional reduced-base case of
\cite[Theorem~1.4]{HMM}; its associated reduced circle bundle
 has real dimension $2\dim_\CC M_G+1\ge3$.  
 The hypotheses of that theorem hold: curvature positivity
is needed only near $P$, and $G$ acts freely there.  Thus there is
$m_0\ge1$ such that}
\begin{equation}\label{eq:eventual}
 A_m\text{ is bijective for every }m\ge m_0.
\end{equation}

\emph{Injectivity in every positive degree.}
Fix $m\ge1$ and suppose $A_m(s)=0$.  Take $k$ with $mk\ge m_0$.
Then \eqref{eq:graded-multiplication} gives $A_{mk}(s^k)=0$.
By \eqref{eq:eventual}, $s^k=0$, hence $s=0$ pointwise.

\emph{Surjectivity in every positive degree.}
Let $0\ne\tau\in H^0(M_G,L_G^m)$ and choose $k\ge1$ with
$mk\ge m_0$.  Lift $\tau^k$ and $\tau^{k+1}$ by
\eqref{eq:eventual} to invariant holomorphic sections
$s_k\in H^0(M,L^{mk})$ and
$s_{k+1}\in H^0(M,L^{m(k+1)})$.
Since their powers restrict identically, injectivity in the positive
degree $mk(k+1)$ yields
\begin{equation}\label{eq:powers-identity}
 s_{k+1}^{\,k}=s_k^{\,k+1}.
\end{equation}
The denominator $s_k$ is not identically zero.  
Its zero set is a proper
analytic subset of the connected complex manifold $M$.  On its
complement the meromorphic section
\begin{equation}\label{eq:consecutive-ratio}
 v=s_{k+1}/s_k\in\operatorname{Mer}(M,L^m)
 \qquad\text{satisfies}\qquad v^k=s_k.
\end{equation}
To check the extension locally, let $e$ be a holomorphic 
frame of $L^m$
and write $s_k=g e^k$, $s_{k+1}=h e^{k+1}$.  The scalar quotient
$f=h/g$ satisfies $f^k=g$.  It is locally bounded up to $\{g=0\}$,
so bounded holomorphic removability
\cite[Lemma~2.3.22]{MaMarinescu} extends it across that
set.  On overlaps the extensions satisfy the section transition law
because they do so on the dense complement of the zero set.  
They give
a global holomorphic section $v$ of $L^m$.  Its $G$-invariance also
extends from the dense invariant complement.  The identity $v^k=s_k$
holds everywhere.

On $P$ where $\tau\ne0$, the ratio restricts to $\tau$.  At any
point where $\tau=0$, $v^k|_P=\tau^k=0$ forces $v|_P=0$ in the
corresponding line fiber.  Thus $A_m(v)=\tau$ everywhere, even on a
component where $\tau$ vanishes identically.  
The zero section has the
\handrev{zero lift.  This proves surjectivity
 in every positive degree.}

\emph{Connectedness and degree zero.}
If $M_G$ had two components, positivity of $L_G$ would give nonzero
sections $\tau_1,\tau_2$ in the same sufficiently high tensor degree,
supported on distinct components.  
Choose the degree high enough to be
very ample on every component.  Their lifts $s_1,s_2$ are nonzero,
while $A_{2m}(s_1s_2)=\tau_1\tau_2=0$ would force $s_1s_2=0$.
Two nonzero holomorphic line-bundle sections on a connected complex
manifold cannot have identically zero product: on a nonempty open set
where the first is nonzero the second would vanish, hence vanish
identically.  This contradiction proves that $M_G$ is connected.
Compactness and connectedness now give
$H^0(M,\OO_M)=H^0(M_G,\OO_{M_G})=\CC$; the map $A_0$ is the identity
on constants.

\emph{Stabilization of the remaining cases.}
\negrev{It remains to prove the nonnegative-degree assertions when
$\dim_\CC M_G=0$.  Replace the data by}
\begin{equation}\label{eq:stabilization}
 M'=M\times\mathbb P^1,\qquad
 L'=L\boxtimes\OO_{\mathbb P^1}(1),
\end{equation}
with the product metric and the trivial $G$-action and lift on the
second factor.  The moment map is the pullback of
$\widetilde\mu$, and its zero set is $P\times\mathbb P^1$.
Regularity, freeness, and positivity near that set all persist; the
curvature is the sum of the two pulled-back curvatures.
The reduced objects are
\begin{equation}\label{eq:stabilized-reduction}
 M'_G=M_G\times\mathbb P^1,\qquad
 L'_G=L_G\boxtimes\OO_{\mathbb P^1}(1).
\end{equation}
Since $\dim_\CC M'\ge2$ and $\dim_\CC M'_G\ge1$, the case
already proved applies, including degree zero and connectedness.
For every $m\ge0$, the K\"unneth formula for global holomorphic
sections identifies the restriction map for the stabilized data with
\begin{equation}\label{eq:tensor-map}
 A_m\otimes\operatorname{id}_{V_m},\qquad
 V_m=H^0(\mathbb P^1,\OO(m)),\qquad \dim V_m=m+1.
\end{equation}
The $G$-action on $V_m$ is trivial.  
Since $V_m$ is nonzero, bijectivity
of \eqref{eq:tensor-map} implies bijectivity of $A_m$, for example by
choosing a basis of $V_m$.  Connectedness of $M'_G$ implies
connectedness of $M_G$ by its surjective projection.  
This completes the proof for all cases in the statement.

Finally \eqref{eq:graded-multiplication} proves that the isomorphisms
assemble into an isomorphism of graded section rings.
\handrev{This completes the proof of Theorem~\ref{thm:line}.}
\end{proof}

\begin{remark}[Stabilization and the one-dimensional obstruction]
\label{rem:stabilization}
\handrev{The proof of \eqref{eq:stabilization} 
applies \cite{HMM} to the
stabilized circle bundle, whose reduced real dimension is at least
three, and recovers the original section maps by tensor cancellation.
This stabilization argument applies only to nonnegative degrees, 
since $V_m=H^0(\mathbb P^1,\OO(m))=0$ for $m<0$.
\negrev{The maximum-norm argument proves negative-degree source
vanishing directly, in every reduced dimension.  Positive reduced
dimension is needed for target-side vanishing and hence for the
negative-degree isomorphisms.}}

\negrev{To see that this restriction is necessary, take
$M=\mathbb P^1$, $L=\OO(2)$ with the Fubini--Study metric, and
$t\cdot[z:w]=[tz:w]$ for $t\in S^1$, twisting the natural lift to $L$
by the character $t$.  
With $\exp(t\xi)=e^{it}$, its moment map is
\[
 \langle\widetilde\mu([z:w]),\xi\rangle
 =\frac1{2\pi}\left(\frac{2|z|^2}{|z|^2+|w|^2}-1\right).
\]
The zero level $P=\{|z|=|w|\}$ is regular and consists of one free
circle orbit, so $M_G$ is a point.  The invariant section $zw$ is
nonvanishing on $P$ and trivializes the descended line $L_G$.
For every $m<0$,
\[
 H^0(\mathbb P^1,\OO(2m))^{S^1}=0,
 \qquad H^0(M_G,L_G^m)\simeq\CC,
\]
so $A_m$ is not surjective. In contrast, for $m\ge0$ the invariant
space is $\CC(zw)^m$ and its canonical descent map is an isomorphism.}
\end{remark}

\subsection{Comparison with classical holomorphic and index results}\label{s4.4}
When $L$ is positive globally and $G$ is connected,
Theorem~\ref{thm:line} recovers the canonical holomorphic quantization
isomorphism of Guillemin--Sternberg \cite{GS82a}
 but not the higher degree
or singular versions of \cite{Zhang99}.  The independent proof above
starts from bijectivity for all
sufficiently large tensor powers in \cite[Theorem~1.4]{HMM}
and obtains local boundedness from the polynomial identity in
\eqref{eq:consecutive-ratio}. 

The theorem of Tian--Zhang \cite[Theorem~0.1]{TZ98} concerns equality
of virtual Spin$^c$ indices for a positive prequantum line bundle.
Applying it to $(L^m,m\omega,m\widetilde\mu)$ gives index identities
in positive tensor degrees.  In the K\"ahler setting, 
vanishing theorems in high tensor degree then recover 
high-degree dimensions of holomorphic sections.
\structrev{An index identity alone does not identify the canonical
$H^0$ map in small degree}: positivity of $L$ does not in
general force every higher cohomology group of $L$ itself to vanish.
The all-degree assertion of \cite{Zhang99} concerns integrable 
holomorphic geometry. 
It seems that \cite[Theorem~6.1]{HMM} does not hold for all
$m\in\mathbb{N}^*$.

\section{Explicit examples}\label{sec:examples}
In Section \ref{subsec:weak-example} we consider an example showing 
how Theorem~\ref{thm:main} applies
when the Levi form degenerates away from the zero level: the
reduction is $S^3$, and smooth CR functions have explicit invariant
lifts.  

In Sections \ref{S:logtube} and \ref{S:projrest} we consider 
the case when $X_G$ has real dimension one. Here the intrinsic CR equation imposes no
condition on functions, and three target-map pairs must 
be distinguished. For the intrinsic target $\mathscr C^\infty(X_G)$,
 restriction and descent have infinite-codimensional range, 
 consisting of smooth paired boundary traces of holomorphic
  functions on an annulus.  For the componentwise
disk-Hardy target of \cite[after (1.2)]{HMM}, the same map does 
not take values in the proposed space.  Fourier projection gives
 a map into that Hardy target, but it is nonmultiplicative and 
 has one-dimensional smooth cokernel.  
 The logarithmic tube in Section \ref{S:logtube} proves these claims and shows why
the dimension hypothesis in Theorem~\ref{thm:main} cannot in general
be omitted.

\subsection{A weakly pseudoconvex hypersurface with 
spherical reduction}\label{subsec:weak-example}
We revisit \cite[\S~2.5, Example~(1)]{HMM} to check the local Levi
hypothesis and to construct a lift explicitly.  The hypersurface and
circle action are
\begin{equation}\label{eq:weak-X}
 \begin{gathered}
 X=\{(z_1,z_2,z_3)\in\CC^3:|z_1|^4+|z_2|^2+|z_3|^2=1\},\\
 e^{i\theta}\cdot z=(e^{-i\theta}z_1,e^{i\theta}z_2,e^{i\theta}z_3).
 \end{gathered}
\end{equation}
Take $\rho=|z_1|^4+|z_2|^2+|z_3|^2-1$ and
$\omega_0=d^c\rho|_X$, where $d^c\rho=-d\rho\circ J$.
The complex Hessian is $\operatorname{diag}(4|z_1|^2,1,1)$.
Thus \(X\) is weakly pseudoconvex, with a degenerate 
Levi direction along \(\{z_1=0\}\).
With the negative fundamental-vector-field
convention \eqref{eq:moment-convention},
\begin{equation}\label{eq:weak-moment}
 \mu=4|z_1|^4-2|z_2|^2-2|z_3|^2,\qquad
 Y=\{|z_1|^4=1/3,\ |z_2|^2+|z_3|^2=2/3\}.
\end{equation}
The $S^1$-action is free on $Y$ and the Levi form is positive nearby. 
 The invariant ambient holomorphic map
\begin{equation}\label{eq:weak-Q}
 Q(z)=\frac{3^{3/4}}{\sqrt2}(z_1z_2,z_1z_3)
\end{equation}
maps $Y$ to $S^3$.  Every orbit in $Y$ has a unique representative
with $z_1=3^{-1/4}>0$, so $Q$ induces 
a diffeomorphism $X_G\simeq S^3$.
Its holomorphicity identifies the reduced CR structure with the
standard one.  The reduction is therefore embeddable and has real
dimension three.

Smooth surjectivity can also be verified explicitly.  Put
$a=|z_1|^4\in[0,1]$.
On $X$,
\begin{equation}\label{eq:weak-bound}
 |Q|^2=\frac{3^{3/2}}2\sqrt a(1-a)\le1,
 \qquad |Q|=1\ \Longleftrightarrow\ a=1/3.
\end{equation}
Indeed the derivative of $\sqrt a(1-a)$ vanishes at $a=1/3$ and
this is its maximum.  Thus $Q(X)\subset\overline{\mathbb B^2}$ and
$Q^{-1}(S^3)=Y$.
For $b\in \handrev{\mathscr C}^\infty(S^3)
\cap\Ker\bar\partial_{b,S^3}$, the
classical ball extension gives
$F\in\OO(\mathbb B^2)\cap \handrev{\mathscr C}^\infty(\overline{\mathbb B^2})$
with boundary value $b$.  \handrev{For this sphere, one can see the
extension directly: the circle Fourier components of a smooth CR
function are restrictions of homogeneous holomorphic polynomials;
their rapid decay in every smooth norm makes the corresponding
polynomial series converge smoothly on the closed ball.}
\handrev{After extending $F$ smoothly past $\overline{\mathbb B^2}$
by the Seeley extension theorem,} the function
\begin{equation}\label{eq:weak-explicit-lift}
 u=F\circ Q|_X
\end{equation}
is smooth and invariant.  It is CR on $X\setminus Y$, where the
composition is holomorphic in a neighborhood, and is CR on all of $X$
by continuity of its tangential derivatives.  The complement is dense
since $Y$ is a regular hypersurface in $X$.  On $Y$ the function
descends to $b$.  This directly verifies smooth surjectivity for this
example; Theorem~\ref{thm:main} supplies uniqueness and the Sobolev
isomorphisms 
\[
 \sigma_{S^1,s}:\Hb(X)^{S^1}_s\overset{\sim}{\longrightarrow}
                          \Hb(S^3)_{s-1/4}.
\]
For the peak-function proof, the functions
\[
 f_p(w)=\frac{1+w_1\overline p_1+w_2\overline p_2}{2},
 \qquad p\in S^3,
\]
peak at $p$ by the equality conditions in the triangle and
Cauchy--Schwarz inequalities.  
Their invariant lifts are $f_p\circ Q$.

\subsection{A logarithmic tube and its reduction}\label{S:logtube}
 To test whether canonical restriction can take values in a
 componentwise disk Hardy space in real dimension one, consider a
 logarithmic tube whose reduction has two circle components.  
Invariant CR monomials will show that restriction and descent
do not take values in the chosen Hardy target.
Let $N\ge2$, write $z_j=e^{x_j+i\theta_j}$ on $(\CC^*)^{N+1}$,
and put
\begin{equation}\label{eq:tube-data}
 \rho=\sum_{j=1}^{N+1}x_j^2-1,\qquad
 D=\{\rho<0\},\quad X=\partial D,\quad
 \omega_0=d^c\rho|_X=2\sum_{j=1}^{N+1}x_jd\theta_j.
\end{equation}
The Hessian is
\[
 \frac{\partial^2\rho}{\partial z_j\partial\bar z_k}
       =\frac{\delta_{jk}}{2|z_j|^2}.
\]
The boundary is therefore strictly pseudoconvex.  It is smooth,
compact, and connected, with
$X\simeq S^N\times(S^1)^{N+1}$ and $\dim_{\RR}X=2N+1\ge5$.
Let $G=(S^1)^N$ rotate the first $N$ coordinates.  The action is free
and preserves $\omega_0$ and the CR structure.  
Its moment map is
\[
 \mu=-2(x_1,\ldots,x_N).
\]
Zero is a regular value, and
\begin{equation}\label{eq:tube-Y}
 Y=\{x_1=\cdots=x_N=0,\ x_{N+1}=\pm1\},\qquad
 X_G=Z_+\sqcup Z_-\simeq S^1\sqcup S^1.
\end{equation}
The induced forms on these two components are
$\pm2d\theta_{N+1}$.  Define
$t_+=\theta_{N+1}$ and $t_-=-\theta_{N+1}$, and let
$\mathscr H_\pm^\infty$ be the smooth nonnegative-mode disk Hardy
spaces in these coordinates.

Let $\mathsf R$ denote restriction to $Y$ followed by descent:
\[
 \mathsf R:\mathscr C^\infty(X)^G\cap\Ker\dbarb
       \longrightarrow\mathscr C^\infty(X_G),\qquad
 (\mathsf R u)([x])=u(x)\quad(x\in Y).
\]
Every $u_k=z_{N+1}^k|_X$, $k\in\mathbb Z$, is a smooth invariant CR
function.  The formula for $\mathsf R$ gives
\begin{equation}\label{eq:tube-restriction}
 \mathsf R(u_k)=\bigl(e^ke^{ikt_+},e^{-k}e^{-ikt_-}\bigr).
\end{equation}
\handrev{The full range into the intrinsic smooth target can also be
identified.  Set $A=\{e^{-1}<|\zeta|<e\}$.  Then
\begin{equation}\label{eq:annulus-range}
 \Ran\mathsf R
 =\bigl\{\bigl(F(e e^{it_+}),F(e^{-1}e^{-it_-})\bigr):
     F\in\OO(A)\cap\mathscr C^\infty(\overline A)\bigr\},
\end{equation}
where smoothness on $\overline A$ means smooth extendibility past its
two boundary circles.  Indeed, expand an invariant 
smooth CR function as
$u=\sum_{k\in\mathbb Z}a_k(x)e^{ik\theta_{N+1}}$, with $x\in S^N$.
Invariance under $G$ makes $u$ independent of
$\theta_1,\ldots,\theta_N$.  For every $v\in T_xS^N$, the tangential
CR vector $L_v=\sum_j v_j(\partial_{x_j}+i\partial_{\theta_j})$
therefore gives
\[
 0=L_vu=\sum_{k\in\mathbb Z}
 \bigl(v\cdot\nabla a_k-kv_{N+1}a_k\bigr)e^{ik\theta_{N+1}},
\]
where $v\cdot\nabla a_k$ denotes the directional derivative of $a_k$
along $v$ on the sphere.  Comparing Fourier coefficients yields
$v\cdot\nabla a_k=kv_{N+1}a_k$, and hence
\(v\cdot\nabla\bigl(e^{-kx_{N+1}}a_k(x)\bigr)=0\).
Since this holds for every tangent vector $v$, connectedness of $S^N$
gives $e^{-kx_{N+1}}a_k(x)=c_k$ for a constant $c_k$.
Smoothness at the two pole circles makes $e^{|k|}c_k$ rapidly
decreasing, so $F(\zeta)=\sum_k c_k\zeta^k$ belongs to the class in
\eqref{eq:annulus-range}; conversely such an $F$ gives a smooth
invariant CR function on $X$.  In the coordinates $t_+,t_-$, every
$(f_+,f_-)\in\Ran\mathsf R$ therefore satisfies
\begin{equation}\label{eq:annulus-fourier}
 \widehat f_-(-k)=e^{-2k}\widehat f_+(k),\qquad k\in\mathbb Z.
\end{equation}
These are infinitely many independent linear constraints on
$\mathscr C^\infty(Z_+)\oplus\mathscr C^\infty(Z_-)$, so
$\Ran\mathsf R$ has infinite codimension.}
For $k>0$ the second component has negative degree; for $k<0$ the
first does.  Thus $\mathsf R$ does not map the invariant CR algebra
into $\mathscr H_+^\infty\oplus\mathscr H_-^\infty$.
Changing orientation alone does not restore this inclusion: on each
component the restrictions of $z_{N+1}$ and $z_{N+1}^{-1}$ are mutual
inverses with nonconstant phase and constant modulus.  Suppose that,
in some smooth parametrization, both restrictions were 
boundary values of holomorphic functions on the disk.  
Their product has boundary value $1$, hence the product of 
the two holomorphic extensions is identically
$1$ by uniqueness of the holomorphic extension.  Each extension is
therefore nonvanishing.  Since its boundary modulus is constant,
$\log|f|$ is harmonic with constant boundary value; thus $|f|$ is
constant on the disk, forcing $f$ itself to be constant, a
contradiction.

\subsection{Projected restriction and its cokernel}\label{S:projrest}
 Since $\mathsf R$ does not take values in the selected Hardy space,
 one may project each
 component onto its nonnegative Fourier modes.  This repairs the
 codomain but sacrifices multiplicativity; the smooth range consists
 precisely of pairs with equal constant Fourier coefficients.
Let $P_\pm$ be the nonnegative Fourier projections and
$\mathsf Q=(P_+\oplus P_-)\mathsf R$.  It is a map to the selected
Hardy target, but by \eqref{eq:tube-restriction},
\[
 \mathsf Q(u_1)=(ee^{it_+},0),\qquad
 \mathsf Q(u_{-1})=(0,ee^{it_-}),\qquad
 \mathsf Q(u_1u_{-1})=(1,1).
\]
Thus it is not multiplicative.  More precisely its smooth range is
\begin{equation}\label{eq:tube-range}
 \Ran\mathsf Q
 =\{(f_+,f_-)\in\mathscr H_+^\infty\oplus\mathscr H_-^\infty:
                      \widehat f_+(0)=\widehat f_-(0)\}.
\end{equation}
\handrev{To prove necessity, average an invariant CR function 
$u$ over rotation of the last coordinate.  
This extracts its zero Fourier mode,
which is constant on $S^N$ by the calculation above with $k=0$.
Its values at the two poles are the constant Fourier coefficients in
\eqref{eq:tube-range}, and projection preserves them.}

Conversely, write a pair with equal constant coefficients as
\[
 f_+=a_0+\sum_{k\ge1}a_ke^{ikt_+},\qquad
 f_-=a_0+\sum_{k\ge1}b_ke^{ikt_-}.
\]
The rapidly decreasing coefficients define the smooth invariant CR
function
\[
 u=a_0+\sum_{k\ge1}e^{-k}
            (a_kz_{N+1}^k+b_kz_{N+1}^{-k})\big|_X.
\]
All derivatives converge uniformly since
$e^{-1}\le|z_{N+1}|\le e$ and each differentiated normalized 
monomial has only polynomial growth in $k$.  Equation
\eqref{eq:tube-restriction} gives $\mathsf Q(u)=(f_+,f_-)$.
This proves \eqref{eq:tube-range}; its smooth cokernel has dimension
one, with missing data represented by the difference of the constants.

\leavevmode\structrev{The residual complex domain is the annulus
$e^{-1}<|\zeta|<e$, so its boundary values on the two circles are
linked rather than arbitrary disk boundary values.}
For $k>0$,
\[
 \mathsf R(e^{-k}u_k)
     =\bigl(e^{ikt_+},e^{-2k}e^{-ikt_-}\bigr).
\]
The negative-frequency component can be exponentially small 
but nonzero. Consequently a local Szeg\H{o} formula or 
an equality modulo smoothing
does not establish the exact global range inclusion required by a
multiplicative restriction map.

\begin{remark}\label{rem:dim-one}
The above example shows that the case where $\dim_{\RR}X_G=1$ must 
be excluded from \cite[Theorem~1.2]{HMM}. The
inclusion 
$\sigma_G\big(H^0_b(X)^G\cap\mathscr C^\infty(X)\big)\subset H^0_b(X_G)$
asserted in \cite[p.~4]{HMM} follows from the CR reduction \cite[(2.43)]{HMM}
when $\dim_{\RR}X_G\ge3$. 
However, when $\dim_{\RR}X_G=1$, we have $HX_G=0$. and 
our definition after \cite[(1.4)]{HMM} of $H^0_b(X_G)$ to be the Hardy space 
is not appropriate, since
the Hardy condition on Fourier coefficients is not implied for the image of $\sigma_G$.
%
For $\dim_{\RR}X_G\ge3$, the only case used here, \cite[Theorem~1.2]{HMM} is
unaffected.
\end{remark}

\section*{Acknowledgements}
Chin-Yu Hsiao was partially supported by National Science and 
Technology Council project 113-2115-M-002-011-MY3.

Rung-Tzung Huang was partially supported by National Science 
and Technology Council project 113-2115-M-008-009-MY3.

Xiaonan Ma was partially supported by Nankai Zhide Foundation,
\handrev{NSFC No.~12621005, Tianjin Science and Technology Program
(No.\ 25ZYZYJC00010)}, and funded through the Institutional 
Strategy of the
University of Cologne within the German Excellence Initiative.

George Marinescu was partially supported by the DFG-funded project
SFB TRR 191.

\smallskip
\noindent
{\it Statement of usage of artificial intelligence.}
The central mathematical ideas and overall approach are 
due to the authors. 
OpenAI ChatGPT was used to assist with proofreading and 
improving the clarity of the exposition.
The authors take full responsibility for the paper’s
content and the accuracy of its references.


\begin{thebibliography}{99}

\bibitem{AG62}
A. Andreotti and H. Grauert,
\emph{Th\'eor\`emes de finitude pour la cohomologie des espaces
 complexes},
Bull. Soc. Math. France \textbf{90} (1962), 193--259.

\bibitem{Atiyah82}
\structrev{M. F. Atiyah,
\emph{Convexity and commuting Hamiltonians},
Bull. London Math. Soc. \textbf{14} (1982), no.~1, 1--15.}

\bibitem{BRT}
M.-S. Baouendi, L.-P. Rothschild, and F. Tr\`eves,
\emph{CR structures with group action and extendability 
of CR functions},
Invent. Math. \textbf{82} (1985), no.~2, 359--396.

\bibitem{BdM75}
L. Boutet de Monvel,
\emph{Int\'egration des \'equations de Cauchy--Riemann induites 
formelles},
S\'eminaire Goulaouic--Lions--Schwartz 1974--1975,
Exp. no.~9, 14 pp., Centre de Math\'ematiques, \'Ecole Polytechnique,
Paris, 1975.

\bibitem{BG}
L. Boutet de Monvel and V. Guillemin,
\emph{The Spectral Theory of Toeplitz Operators},
Annals of Mathematics Studies, vol.~99,
Princeton University Press, Princeton, NJ, 1981.

\bibitem{BMS}
L. Boutet de Monvel and J. Sj\"ostrand,
\emph{Sur la singularit\'e des noyaux de Bergman et de Szeg\H{o}},
Ast\'erisque \textbf{34--35} (1976), 123--164.

\bibitem{Braverman01}
M. Braverman,
\emph{Cohomology of the Mumford quotient},
in \emph{Quantization of Singular Symplectic Quotients},
Progress in Mathematics, vol.~198, Birkh\"auser, Basel, 2001,
47--59.

\bibitem{Bur79} D.\ M.\ Burns, 
\emph{Global behavior of some tangential {C}auchy-{R}iemann  equations}, 
Partial differential equations and geometry (Proc. Conf., Park
  City, Utah, 1977), Lecture Notes in Pure and Appl. Math., vol.~48, 
Dekker, New York, 1979, 51--56.

\bibitem{GS82convexity}
\structrev{V. Guillemin and S. Sternberg,
\emph{Convexity properties of the moment mapping},
Invent. Math. \textbf{67} (1982), no.~3, 491--513.}

\bibitem{GS82a}
V. Guillemin and S. Sternberg,
\emph{Geometric quantization and multiplicities of
 group representations},
Invent. Math. \textbf{67} (1982), 515--538.

\bibitem{HL75}
F. R. Harvey and H. B. Lawson, Jr.,
\emph{On boundaries of complex analytic varieties. I},
Ann. of Math. (2) \textbf{102} (1975), 223--290.

\bibitem{HLadd}
F. R. Harvey and H. B. Lawson, Jr.,
\emph{Addendum to Theorem 10.4 in ``Boundaries of Analytic Varieties''},
2000, arXiv:math/0002195.  This addendum contains the amended
Theorem~10.4$'$ used here.

\bibitem{HHMS}
H. Herrmann, C.-Y. Hsiao, G. Marinescu and W.-C. Shen, 
\emph{Semi-classical spectral asymptotics of Toeplitz operators on CR manifolds},
arXiv:2303.17319v2.


\bibitem{HillNacinovich}
C. D. Hill and M. Nacinovich,
\emph{Stein fillability and the realization of contact manifolds},
Proc. Amer. Math. Soc. \textbf{133} (2005), no.~6, 1843--1850.


\bibitem{HochsSong17}
P. Hochs and Y. Song,
\emph{Equivariant indices of Spin$^c$-Dirac operators for 
proper moment maps},
Duke Math. J. \textbf{166} (2017), no.~6, 1125--1178.

\bibitem{HHtorsion}
C.-Y. Hsiao and R.-T. Huang,
\emph{The asymptotics of the analytic torsion on CR manifolds with
$S^1$ action},
Commun. Contemp. Math. \textbf{21} (2019), no.~4, 1750094, 35 pp.


\bibitem{HH}
C.-Y. Hsiao and R.-T. Huang,
\emph{$G$-invariant Szeg\H{o} kernel asymptotics and CR reduction},
Calc. Var. Partial Differential Equations \textbf{60} (2021), 
Paper No.~47, 48 pp.

\bibitem{HMM}
C.-Y. Hsiao, X. Ma, and G. Marinescu,
\emph{Geometric quantization on CR manifolds},
Commun. Contemp. Math. \textbf{25} (2023), no.~10, 2250074, 73 pp.

\bibitem{HM} C.-Y. Hsiao and G. Marinescu, 
\emph{On the singularities of the Szeg\H{o} projections on 
lower energy forms}, J. Diff. Geom. \textbf{107} (2016), 83-155.

\bibitem{Kirwan84}
\structrev{F. C. Kirwan,
\emph{Cohomology of Quotients in Symplectic and Algebraic Geometry},
Mathematical Notes, vol.~31,
Princeton University Press, Princeton, NJ, 1984.}

\bibitem{Kohn86}
J. J. Kohn,
\emph{The range of the tangential Cauchy--Riemann operator},
Duke Math. J. \textbf{53} (1986), 525--545.

\bibitem{MaICM10}
X. Ma,
\emph{Geometric quantization on K\"ahler and symplectic manifolds},
in \emph{Proceedings of the International Congress of Mathematicians}
(Hyderabad, 2010), vol.~II, Hindustan Book Agency, New Delhi, 2010,
pp.~785--810.

\bibitem{MaMarinescu}
X. Ma and G. Marinescu,
\emph{Holomorphic Morse Inequalities and Bergman Kernels},
Progress in Mathematics, vol.~254, Birkh\"auser, Basel, 2007.

\bibitem{MZ08}
X. Ma and W. Zhang,
\emph{Bergman kernels and symplectic reduction},
Ast\'erisque \textbf{318} (2008), 154 pp.


\bibitem{MaZhang09}
X. Ma and W. Zhang,
\emph{Geometric quantization for proper moment maps},
C. R. Math. Acad. Sci. Paris \textbf{347} (2009), no.~7--8, 389--394.

\bibitem{MZ14}
X. Ma and W. Zhang,
\emph{Geometric quantization for proper moment maps: the Vergne
conjecture},
Acta Math. \textbf{212} (2014), 11--57.

\bibitem{MY07}
\handrev{G. Marinescu and N. Yeganefar,
\emph{Embeddability of some strongly pseudoconvex CR manifolds},
Trans. Amer. Math. Soc. \textbf{359} (2007), no.~10, 4757--4771.}

\bibitem{Meinrenken96}
E. Meinrenken,
\emph{On Riemann--Roch formulas for multiplicities},
J. Amer. Math. Soc. \textbf{9} (1996), 373--389.

\bibitem{Meinrenken98}
E. Meinrenken,
\emph{Symplectic surgery and the Spin$^c$-Dirac operator},
Adv. Math. \textbf{134} (1998), 240--277.

\bibitem{MeinrenkenSjamaar99}
E. Meinrenken and R. Sjamaar,
\emph{Singular reduction and quantization},
Topology \textbf{38} (1999), no.~4, 699--762.

\bibitem{NN}
A. Newlander and L. Nirenberg,
\emph{Complex analytic coordinates in almost complex manifolds},
Ann. of Math. (2) \textbf{65} (1957), 391--404.

\bibitem{Paradan01}
P.-E. Paradan,
\emph{Localization of the Riemann--Roch character},
J. Funct. Anal. \textbf{187} (2001), no.~2, 442--509.


\bibitem{Paradan11}
P.-E. Paradan,
\emph{Formal geometric quantization II},
Pacific J. Math. \textbf{253} (2011), 169--211.

\bibitem{ParadanVergne17}
\handrev{P.-E. Paradan and M. Vergne,
\emph{Equivariant Dirac operators and differentiable geometric 
invariant theory},
Acta Math. \textbf{218} (2017), no.~1, 137--199.}

\bibitem{Starcic}
T. Star\v{c}i\v{c},
\emph{Stein neighborhood bases of embedded strongly pseudoconvex
 domains and approximation of mappings},
J. Geom. Anal. \textbf{18} (2008), no.~4, 1133--1158.

\bibitem{Teleman00}
C. Teleman,
\emph{The quantization conjecture revisited},
Ann. of Math. (2) \textbf{152} (2000), no.~1, 1--43.

\bibitem{TZ98}
Y. Tian and W. Zhang,
\emph{An analytic proof of the geometric quantization conjecture of
Guillemin--Sternberg},
Invent. Math. \textbf{132} (1998), 229--259.

\bibitem{Vergne96}
M. Vergne,
\emph{Multiplicities formula for geometric quantization. I, II},
Duke Math. J. \textbf{82} (1996), 143--179, 181--194.

\bibitem{VergneICM06}
M. Vergne,
\emph{Applications of equivariant cohomology},
in \emph{Proceedings of the International Congress of Mathematicians}
(Madrid, 2006), vol.~I, European Mathematical Society, Z\"urich, 2007,
pp.~635--664.

\bibitem{Zhang99}
W. Zhang,
\emph{Holomorphic quantization formula in singular reduction},
Commun. Contemp. Math. \textbf{1} (1999), no.~3, 281--293.

\end{thebibliography}
\end{document}